\documentclass[12pt]{amsart}
\usepackage[utf8]{inputenc}
\usepackage{amsfonts,amsmath,amssymb,amsxtra,stmaryrd,mathdots}

\usepackage{a4wide}

\usepackage{colonequals}
\usepackage[utf8]{inputenc}
\usepackage[british]{babel}
\usepackage{hyperref}
\hypersetup{colorlinks=true,urlcolor=blue,citecolor=blue,linkcolor=blue}

\usepackage[all,cmtip]{xy}
\usepackage{float}

\usepackage[dvipsnames]{xcolor}
\usepackage[normalem]{ulem}
\usepackage{cleveref}
\Crefname{equation}{}{}

\usepackage{tikz}
\usepackage{tikz-cd}

\usepackage{amsmath}
\usepackage{mathtools}
\usepackage{thmtools}
\usepackage{thm-restate}
\usepackage{enumitem}

\usepackage{tikz-cd}
\usetikzlibrary{patterns,shapes,cd,arrows}
\usepackage[auto]{contour}
\contourlength{2pt}

\newcommand{\Q}{\mathbb{Q}}

\newcommand{\C}{\mathbb{C}}

\newcommand{\Z}{\mathbb{Z}}

\DeclareMathOperator{\End}{\operatorname{End}}

\DeclareMathOperator{\GL}{GL}

\DeclareMathOperator{\Gal}{Gal}

\DeclareMathOperator{\GSp}{GSp}

\DeclareMathOperator{\Aut}{Aut}

\DeclareMathOperator{\MT}{MT}
\DeclareMathOperator{\ST}{ST}

\DeclareMathOperator{\Jac}{Jac}

\DeclareMathOperator{\HH}{H}

\newcommand{\et}{\text{ét}}

\newcommand{\KconnJ}{K(\varepsilon_{J_m})}
\newcommand{\KconnA}{K(\varepsilon_A)}

\newcommand{\gammaAHa}{\gamma}

\usepackage{mathrsfs}

\numberwithin{equation}{subsection}

\newtheorem{theorem}[equation]{Theorem}

\newtheorem{corollary}[equation]{Corollary}

\newtheorem{lemma}[equation]{Lemma}

\newtheorem{proposition}[equation]{Proposition}

\theoremstyle{definition}
\newtheorem{definition}[equation]{Definition}

\theoremstyle{remark}
\newtheorem*{remark*}{Remark}
\newtheorem{remark}[equation]{Remark}

\newcommand{\Gl}{\mathcal{G}_\ell}

\newcommand{\Gam}[2]{\Gamma\left(\frac{#1}{#2}\right)}

\title[Monodromy groups and exceptional Hodge classes]{Monodromy groups and exceptional Hodge classes, II: Sato-Tate groups}

\author{Andrea Gallese, Heidi Goodson, and Davide Lombardo}

\begin{document}

\begin{abstract}
Denote by $J_m$ the Jacobian variety of the hyperelliptic curve defined by the affine equation $y^2=x^m+1$ over $\mathbb{Q}$, where $m \geq 3$ is a fixed positive integer.
In this paper, we compute the Sato-Tate group of $J_m$.
Currently, there is no general algorithm that computes this invariant. 
We also describe the Sato-Tate group of an abelian variety, generalizing existing results that apply only to non-degenerate varieties, and prove an extension of a well-known formula of Gross-Koblitz that relates values of the classical and $p$-adic gamma functions at rational arguments.
\end{abstract}

\maketitle
\let\thefootnote\relax
\let\thefootnote\relax

\setcounter{tocdepth}{2}
\tableofcontents

\section{Introduction}
This paper is a continuation of the work described in \cite{part1}. In both papers, we investigate how the existence of exceptional algebraic cycles (or, more precisely, exceptional Hodge classes) on an abelian variety $A$ affects some of its associated invariants, both for general abelian varieties and for a specific family of interesting examples.

In complex algebraic geometry, an \emph{exceptional} Hodge class is an element of the Hodge ring not generated by divisor classes. 
Abelian varieties %
that support such exceptional classes are called \emph{degenerate}. While the definition of degeneracy is a statement about the Hodge ring, we can see the effects of degeneracy in several
invariants of the abelian variety $A$: the Mumford-Tate group and, assuming $A$ is defined over a number field, the connected monodromy field and the Sato-Tate group (see below for definitions).

By definition, the Mumford-Tate group is related to the Hodge ring via its action on the cohomology of $A$. On the other hand, the connected monodromy field and the Sato-Tate group are arithmetic invariants constructed from the $\ell$-adic monodromy groups of $A$, which are themselves obtained from the images of the various $\ell$-adic Galois representations associated with $A$.
Under the assumption of the algebraic Sato-Tate and Mumford-Tate conjectures, 
all these invariants 
are related to one another, which leads to the conclusion that exceptional Hodge classes should have a direct influence on the Galois representations attached to $A$. The rest of this introduction describes these notions in greater detail, introduces our main family of examples, reviews the key inputs needed from \cite{part1}, and states our new results.

\subsubsection*{Galois representations and Mumford-Tate groups}
Let $A$ be an abelian variety defined over a number field $K$. Fix an auxiliary prime number $\ell$ and an algebraic closure $\overline{K}/K$, and denote by $\Gamma_K$ the corresponding absolute Galois group. The associated Galois representation is the homomorphism $$\rho_{A, \ell}\colon \Gamma_K \to \Aut(V_\ell A)$$ induced by the natural action of the Galois group on the Tate module $V_\ell A = \varprojlim_n A[\ell^n] \otimes_{\Z_\ell} \Q_\ell$. The \emph{$\ell$-adic monodromy group} $\Gl$ is the Zariski closure of $\operatorname{im}(\rho_{A, \ell}) \subseteq \GL_{V_\ell A}$. The group of components of the monodromy group is naturally isomorphic to the Galois group of a finite extension $K(\varepsilon_A)/K$, which does not depend on $\ell$ (see \cite[Section 2.1]{part1}), called the \emph{connected monodromy field}. 

We now briefly describe the Mumford-Tate group and the Mumford-Tate conjecture. Fix an embedding of $K$ into the complex numbers, so that we can take the $\mathbb{C}$-points of $A$. By Hodge theory, the $\mathbb{Q}$-vector space $\operatorname{H}^1(A(\mathbb{C}), \mathbb{Q})$ carries a natural pure Hodge structure of weight 1. In general, to any such Hodge structure one can attach a \textit{Mumford-Tate group}, which is by definition a connected $\mathbb{Q}$-algebraic subgroup of $\operatorname{GL}_V$ (see \cite[Section 2]{part1} and the references therein). The Mumford-Tate group of $A$, denoted $\operatorname{MT}(A)$, is the Mumford-Tate group of the Hodge structure $V$.

The Mumford-Tate conjecture is a comparison statement between the groups $\Gl$ and the group $\operatorname{MT}(A)$: it predicts that, in a sense made precise by the Betti-to-étale comparison theorem, for each prime $\ell$ we have a natural isomorphism $\operatorname{MT}(A) \times_{\Q} \Q_\ell \cong \Gl^0$. In other words, the Mumford-Tate conjecture provides a (conjectural) candidate group that interpolates the identity components of the various $\Gl$. In the next section, we discuss the Sato-Tate group, which (conjecturally) gives a way to interpolate not just the identity component, but also the entire group $\Gl$, as $\ell$ varies.

\subsubsection*{Sato-Tate groups} 
As anticipated in the previous section, the algebraic Sato-Tate conjecture predicts that, roughly speaking, the (possibly disconnected) $\ell$-adic monodromy groups $\Gl$ for varying $\ell$ are all interpolated by a single algebraic group defined over $\Q$, called the \textit{algebraic Sato-Tate group} $\operatorname{AST}(A)$. More precisely, it predicts the existence of a $\Q$-algebraic group $\operatorname{AST}(A)$ such that, for every prime $\ell$, we have canonical isomorphisms $\operatorname{AST}(A) \times_{\Q} \Q_\ell \cong \Gl$, where the isomorphism is again induced by the comparison isomorphism between étale and Betti cohomology.

Cantoral-Farf\'{a}n and Commelin \cite{MR4530050} have shown that the algebraic Sato-Tate conjecture is true for all abelian varieties that satisfy the Mumford-Tate conjecture, hence in particular for all CM abelian varieties, since the Mumford-Tate conjecture is known in the CM case \cite{Pohlmann}.
When the algebraic Sato-Tate conjecture holds, one defines the Sato-Tate group $\operatorname{ST}(A)$ as a maximal compact subgroup of the complex points of $\operatorname{AST}(A)$ (see \Cref{def: Sato-Tate group} for details).

Much of the interest in the Sato-Tate group comes from the (generalized) Sato-Tate conjecture, which predicts the asymptotic distribution of characteristic polynomials of Frobenius of an abelian variety $A$ in terms of $\ST(A)$. This conjecture has been proven for elliptic curves with complex multiplication (CM) and for those defined over totally real number fields (see \cite{Barnet2011, Clozel2008,Harris2010,Taylor2008}), as well as for abelian varieties of arbitrary dimension with potential complex multiplication \cite[Proposition 16]{Joh17}.

At the moment, there is no general strategy to compute the Sato-Tate group of a given abelian variety though there has recently been progress for small dimension and for nondegenerate abelian varieties. Classification results for the Sato-Tate groups of dimension 2 and 3 abelian varieties are given in \cite{Fite2012} and \cite{fitekedlayasuth2023}. 
For computations of the Sato-Tate groups of nondegenerate abelian varieties, see, for example, \cite{Arora2016, EmoryGoodson2022,EmoryGoodson2024, FGL2016, FiteLorenzo2018, GoodsonCatalan,GoodsonHoque2024,KedlayaSutherland2009, LarioSomoza2018}.

Many of these results were made possible by work of Banaszak and Kedlaya, who proved \cite[Theorem 6.1]{MR3502937} that the algebraic Sato-Tate group of a \emph{stably non-degenerate} abelian variety coincides with its twisted (decomposable) Lefschetz group $\operatorname{TL}_A$ \cite[Definition 3.4]{MR3502937}. This is the subgroup of $\GL_V$, where as above $V$ is the first Betti cohomology of $A(\mathbb{C})$, 
given by those automorphisms that preserve the polarisation up to scalars and commute with the endomorphisms of $A$ up to the Galois action. %
However, this method %
does not extend to degenerate abelian varieties. This is one of the limitations we aim to overcome in this paper: we describe a generalization of the twisted Lefschetz group that takes into account all Hodge classes, not just the endomorphisms and the polarization.

\subsubsection*{Fermat Jacobians} We focus in particular on a family of abelian varieties that contains many degenerate ones. 
Specifically, we study the Jacobians $J_m$ of the hyperelliptic curves
$$C_m: y^2=x^m+1,$$
where $m\geq 3$ is any integer. Such abelian varieties are known to be degenerate if, for example, $m$ is an odd composite number (see \cite[Theorem 1.1]{Heidi}) and are nondegenerate if $m$ is a power of 2 (\cite[Theorem 1.2]{EmoryGoodson2024}), prime (\cite{Shimura}), or twice an odd prime (\cite{EmoryGoodson2022}). Moreover, the algebraic Sato-Tate and Mumford-Tate conjectures are known to hold for these varieties, so that in particular we know unconditionally that exceptional Hodge classes in the cohomology of $J_m$ have a close relationship with its Galois representations.
We also note that the curves $C_m$ arise as quotients of the Fermat curves $x^m + y^m + z^m = 0 \subset \mathbb{P}_{2, \mathbb{Q}}$, which is why we call $J_m$ a \textit{Fermat Jacobian}.

Working with these concrete examples allows us to demonstrate different phenomena that can occur for degenerate abelian varieties and develop methods for understanding their structure. In particular, we derive from our general description of Sato-Tate groups an algorithm to compute the Sato-Tate group of any Fermat Jacobian $J_m$. The algorithm is practical and we implement it to fully describe the first known example of the Sato-Tate group of a degenerate abelian variety.

\begin{remark*}
According to the Hodge conjecture, every Hodge class (in particular, every exceptional class) should arise from an algebraic cycle. While we don't explicitly construct exceptional algebraic \emph{cycles} in this paper, we do exhibit exceptional \emph{classes} and pin down precisely the cohomology groups in which they live. Our methods also allow us to compute the field of definition of these classes.
\end{remark*}

\subsubsection*{Review of part I}
We recall some of the notation that we fixed in part I and some useful results. 
For a complex abelian variety $A$, denote by $V = \HH^1(A, \Q)$ the first singular cohomology group. The Mumford-Tate group $\operatorname{MT}(A)$ is the maximal subgroup of $\GL_V$ fixing all Hodge classes in all Hodge structures $\bigoplus_{i} V^{\otimes 2n_i}(n_i)$, where $\{n_i\}$ is any finite sequence of positive integers (see \cite[Section 2.2]{part1}).

In \cite[Section 4]{part1}, we gave an algorithm to compute the Mumford-Tate group of the abelian variety $J_m$ (see \cite{OurScripts} for an implementation).
The output is a finite set $f_1, \, \dots, \, f_r$ of \emph{equations for $\MT(J_M)$} (see \cite[Definition 4.2.4]{part1}); these are Laurent monomials $f_i= \prod x_{j}^{e_{j}}$ such that the equations $f_i=1$ cut out $\MT(J_m)_{\mathbb{C}}$ as an algebraic subvariety of the diagonal torus of $\GL_{V \otimes \mathbb{C}}$.
This data is equivalent to a set of generators for the algebra of Tate classes (see \Cref{lemma: degree of equations is the same as degree of classes}).

\smallskip

One of the main results of \cite{part1} is the computation of the connected monodromy field $\Q(\varepsilon_{J_m})/\Q$. We highlight the main steps of this computation. Assume that $m \geq 3$ is an odd integer, for the sake of simplicity, and fix a prime $\ell \mid m-1$, so that we can embed $K=\Q(\zeta_m)$ into $\Q_\ell$. In what follows, for simplicity of notation, given a vector space $W$ we denote by $T^nW$ the $n$-th tensor power of $W$. If $W$ is a representation of a group (or algebraic group) $\mathcal{G}$, we consider $T^nW$ with its natural structure as a representation of $\mathcal{G}$.

\begin{definition}
    {\label{remark: the hyperelliptic ladic basis}}
    We fix a basis $\{v_i\}$ for $V_\ell$ with the following property. Let $\alpha_m\colon J_m \to J_m$ denote the endomorphism induced by $C_m \to C_m, \, (x,y) \mapsto (\zeta_mx, y)$. Then $\alpha_m^\ast v_i = \zeta_m^i \cdot v_i$ for $i =1, \dots, 2g$, i.e.~$v_i$ is an ordered eigenbasis for the linear operator $\alpha_m^\ast\colon V_\ell \to V_\ell$.

    Such a basis exists when all (different) eigenvalues are defined over $\Q_\ell$; this is the reason why it is easier to work with primes $\ell$ dividing $m-1$.
    The usual hyperelliptic basis $\omega_i = x^{i-1}\, dx/y$ in de Rham cohomology satisfies this property (see \cite[Section 2.4]{part1} for a more complete introduction). Every $\omega_i$ is, in fact, a scalar multiple of the corresponding $v_i$, after passing through the de Rham-to-étale comparison isomorphism. 
\end{definition}

Let $f = \prod x_j^{e_j}$ be an equation for $\MT(J_m)$.
Denote by $q$ the sum of its negative exponents (that is, $q := -\sum_{i} \min\{e_i,0\}$), let $g$ be the genus of $C_m$, and set $2p = n = 4qg$.
We define a class
\begin{equation}
    {\label{eq: associate tate classes vf}}
    v_f = v_1^{\otimes e_1} \otimes\cdots\otimes  v_{2g}^{\otimes e_{2g}} \otimes (2\pi i)^p
\in T^{n}V_\ell J_m (p).
\end{equation}

The next proposition describes the connected monodromy field of $J_m$ in terms of the classes $v_f$. 

\begin{proposition}[{\cite[{Proposition 4.3.3}]{part1}}]{\label{prop:fixMT}}
    Fix a finite set of equations ${f}_i$ defining the Mumford-Tate group $\MT(J_m)$. Let $K$ be a number field containing all $m$-th roots of unity. The corresponding $v_{f_i}$ are Tate classes and the {minimal {extension of $K$} over which all} $v_{f_i}$ {are simultaneously defined} is the connected monodromy field $K(\varepsilon_{J_m})$.
\end{proposition}

For the next step of our computation, we need to understand the action of $\operatorname{Gal}(\overline{\Q}/\Q(\zeta_m))$ on these Tate classes. In order to do so, we embed the {Tate classes on $J_m$ and its powers} into the {cohomology} ring of a suitable Fermat hypersurface $X_m^n$ (which is to say, for a suitable $n$), where
\[
X_m^n : x_0^m + x_1^m + \cdots + x_{n+1}^m =0 \subset \mathbb{P}_{n+1, \Q}.
\]
The abelian group $G_m^n = \mu_m^{n+1}/\Delta\mu_m$ (quotient by the image of the diagonal embedding) acts naturally on $X_m^n$. The étale cohomology algebra of $X_m^n$ decomposes as the direct product of one-dimensional representations, indexed by the character group of $G_m^n$.

\begin{remark}
    {\label{remark: character indexing}}
    A character $\alpha$ for $G_m^n$ can be represented by an $(n+2)$-tuple of exponents $(a_0, \dots, a_{n+1}) \in (\Z/m\Z)^{n+2}$ such that $\sum_{i=0}^{n+1} a_i = 0$.
    The set $\mathfrak{B}_m^n$ of characters whose eigenvectors are Tate classes is characterized by the following condition: $\alpha \in \mathfrak{B}_m^n$ if and only if all $a_i$ are non-zero and
    \begin{equation}\label{eq: Amn and Bmn}
    \textstyle \langle t \alpha\rangle \colonequals \sum_i [t a_i] /m =n/2+1  \text{ for all }t \in (\Z/m\Z)^\times,
    \end{equation}
    where $[a] \in \mathbb{Z}$ is the only representative for $a \in \Z/m\Z$ lying in the interval $[0,m-1]$.
    The characters $\gamma_i = (i, i , -2i)$ of $G_m^1$ play a central role in our work. Given two characters $\gamma_1 \leftrightarrow (a_0, \ldots, a_{n_1+1})$ of $G_m^{n_1}$ and $\gamma_2 \leftrightarrow (b_0, \ldots, b_{n_2+1})$ of $G_m^{n_2}$, we will denote by $\gamma_1 \ast \gamma_2$ the character of $G_m^{n_1+n_2+2}$ corresponding to $(a_0, \ldots, a_{n_1+1}, b_0, \ldots, b_{n_2+1})$. Finally, we will denote by $\gamma_1^{\ast k}$ the concatenation $\underbrace{\gamma_1 \ast \cdots \ast \gamma_1}_{k\text{ times}}$.

    Given a cohomology group $\HH$ for $X_m^n$ and a character $\alpha$, we denote by $\HH_\alpha$ the corresponding eigenspace.
\end{remark}

As already indicated, the next step in our strategy is to embed the Tate classes that we constructed into the cohomology of a suitable Fermat hypersurface. More precisely, with the above notation we have the following.

\begin{proposition}[{\cite[{Lemma 5.2.3 and Propositions 5.2.6, 5.3.13}]{part1}}]{\label{prop:ShiodaIII}}
    With the above notation, let $h = 6p-2$. 
    \begin{enumerate}
        \item There is a $\Gamma_{\Q}$-equivariant embedding
        \[ \psi^\ast\colon \HH^1_{\et}({C}_{m}, \Q_\ell) \longrightarrow \HH^1_{\et}(X_m^1, \Q_\ell). \]
        This isomorphism sends $v_i$ to an element $\psi^\ast(v_i)$ in the ${\gamma_i}$-eigenspace.

        \item There is a $\Gamma_\Q$-equivariant inclusion
        \begin{equation}{\label{eq: theta}}
            \Theta \colon T^{2p}V_\ell J_m(p) \hookrightarrow \HH_{{\et}}^{h}(X_m^h, \Q_\ell)(h/2).
        \end{equation}
        The vector $\Theta(v_f)$ is a Tate class belonging to the $G_m^{h}$-eigenspace {with character $\gamma_f$}, where
        $\gamma_f = \gamma_1^{\ast e_1}\ast \dots \ast \gamma_{2g}^{\ast e_{2g}} \in \mathfrak{B}_m^h$. 
    \end{enumerate}
\end{proposition}

\begin{remark}
{\label{remark: twists}}
    Several technical difficulties have been hidden in the statement of \Cref{prop:ShiodaIII}. An important step in defining $\Theta$ is that we identify the cohomology of $J_m$ with that of a specific twist $\Tilde{J}_m$ (the Jacobian of $y^2=1-4x^m$, which is dominated by the Fermat curve $X^1_m$ over $\Q$). We prove in \cite[Section 5.3]{part1} that, when $m$ is odd, the twisting isomorphism $J_m \to \Tilde{J}_m$ induces a Galois-equivariant isomorphism in cohomology. This is not always true when $m$ is even, but it is possible to keep track of the twisted Galois action.
    For the sake of simplicity, in this paper we always assume that $m$ is odd, although all arguments could easily be adapted to account for a twist of the Galois action.
\end{remark}

Finally, we exploit an explicit description of the Galois action (of $\Gal(\overline{\Q}/\Q(\zeta_m))$ on the Tate classes of $X_m^n$ in terms of the characters in $\mathfrak{B}_m^n$, due to Deligne.
Denote by $\Gamma(s)$ the classical Euler Gamma function. We associate with $f = x_1^{e_1} \cdots x_{m-1}^{e_{m-1}}$ its Gamma-value
    \begin{equation}{\label{eq:Gammaf}}
    \Gamma(f) \colonequals \prod_{j=1}^{m-1} \left[
    \Gamma\left(\tfrac{j}{m}\right) \cdot \Gamma\left(\tfrac{j}{m}\right) \cdot 
    \Gamma\left(\tfrac{[-2j]}{m}\right) \right]^{e_j},    
    \end{equation}
where $[-2j]$ is the integer between $1$ and $m$ which is congruent to $-2j$ modulo $m$.

\begin{theorem}[{\cite[Theorem 6.4.5]{part1}}]
{\label{cor: connected monodromy field for odd m in terms of equations}}
    Let $m \geq 3$ be an odd positive integer and $J_m / {\Q}$ be the Jacobian of the smooth projective curve {over $\Q$} with affine equation $y^2=x^m+1$.
    Let $f_1,\, f_2, \,  \dots,\, f_r$ be a finite set of equations for $\MT(J_m)$. The field $\Q(\varepsilon_{J_m})$ is generated over $\Q$ by the (algebraic) complex numbers
    $$\zeta_m,\, \Gamma(f_1),\, \Gamma(f_2), \,  \dots, \,\Gamma(f_r).$$
\end{theorem}
This concludes the computation of the connected monodromy field of $J_m$. Many of the constructions that we have just reviewed will also play an important role in this paper.

\subsubsection*{Computation of the Sato-Tate group} 
We now come to the main new results of this paper.
In \Cref{section:Sato-Tate} we characterize the monodromy group $\Gl$ as the subgroup of $\GL$ respecting all Tate classes up to Galois action. More precisely, letting $V_\ell$ denote the $\ell$-adic Tate module of $A$, we identify a finite-dimensional subspace $W_{\leq N}$ of the tensor algebra $T^\bullet(V_\ell \otimes V_\ell^\vee)$, generated by Tate classes, such that the following proposition holds.

\begin{restatable*}{theorem}{DescriptionOfGl}
\label{thm: description of Gl in terms of Tate classes}
Let $A$ be an abelian variety over a number field $K$.
Let $N$ be as in \Cref{cor: algebra of Tate classes is finitely generated} and $W_{\leq N}$ be the tensor representation of \Cref{definition: Space of Tate classes of degree up to r W} (carrying the Galois action given by \eqref{eq: representation of the Galois group of Kconn in W truncated at r}). We have
\[
\mathcal{G}_\ell(\Q_\ell) = \bigsqcup_{\sigma \in \Gal(\KconnA/K)} \left\{ h \in \GL_{V_\ell}(\Q_\ell) \bigm\vert \begin{array}{c}
h(W_{\leq N}) \subseteq W_{\leq N} \\
h|_{W_{\leq N}} = \rho_{W_{\leq N}}(\sigma)
\end{array} \right\}.
\]
\end{restatable*}
This description holds for any abelian variety over a number field. 
Although probably known to experts, this result does not seem to appear explicitly in the literature and helps to illuminate the role of nondegeneracy in the work of Banaszak and Kedlaya \cite{MR3502937}. The main improvement in \Cref{thm: description of Gl in terms of Tate classes} is that $W_{\leq N}$ contains a generating set of Tate classes, not only those arising from the endomorphisms of $A$ and its polarization, which are the only ones taken into account by the twisted Lefschetz group.
In particular, in \Cref{rmk: relation with the twisted Lefschetz group} we compare this result with the description of the Sato-Tate group of stably nondegenerate abelian varieties given by Banaszak and Kedlaya in \cite{MR3502937} and show how our description reduces to theirs under suitable additional assumptions. Thus, once again, we see how the existence of exceptional Tate classes (which, under Mumford-Tate, correspond to exceptional Hodge classes) directly impacts the arithmetic invariants of $A$.
We also remark that -- by definition -- $\mathcal{G}_\ell$ determines the Sato-Tate group of $A$, so \Cref{thm: description of Gl in terms of Tate classes} implicitly also describes $\operatorname{ST}(A)$.

\subsubsection*{The Sato-Tate group of a Fermat Jacobian} 
Regarding our main family of examples, the Jacobian varieties $J_m$ %
were studied by Shioda in \cite{Shioda3}, where he proved the Hodge conjecture for $J_m$ for infinite families of values of $m$, including some for which $J_m$ is degenerate. It can be shown that in most cases the Hodge ring of $J_m$ contains exceptional Hodge cycles (see \cite[Section 6]{Shioda3} and \cite{Heidi}). 
The Sato-Tate groups of certain varieties $J_m$ (and twists thereof) were described in work of Fité and Sutherland, see \cite{MR3218802} for $m=6$ and \cite{MR3502940} for $m=8$.
We also remark that the Sato-Tate group of the Jacobian of Fermat hypersurfaces over $\Q(\zeta_m)$ has been computed by Lorenzo García in \cite[Section 5]{elisa_thesis}. 

These results, however, don't provide a uniform framework for computing $\operatorname{ST}(J_m)$ for arbitrary $m$, and all fall within the context of nondegenerate abelian varieties (all abelian varieties of dimension up to 3 are nondegenerate).
In this article, we expand on the earlier work and solve completely the problem of describing the Sato-Tate group of $J_m$ for every $m$. The crucial ingredient is the determination of the Tate classes (the analogue in étale cohomology of Hodge classes) on the varieties $J_m$ and their powers, together with the action of $\Gal\left(\overline{\Q}/\Q\right)$ on them.

In the case of the Fermat Jacobian $J_m$, exploiting our knowledge of Tate classes, we make the characterization of \Cref{thm: description of Gl in terms of Tate classes} completely explicit, which enables us to compute the Sato-Tate group of $J_m$. {To do so}, we embed the {Tate classes on $J_m$ and its powers} into the {cohomology} ring of a suitable Fermat hypersurface $X_m^n$ (which is to say, for a suitable $n$), where
\[
X_m^n : x_0^m + x_1^m + \cdots + x_{n+1}^m =0 \subset \mathbb{P}_{n+1, \Q}.
\]
Here, the subspace $W_{\leq N}$ decomposes into one-dimensional representations of a certain group $G_m^n$, whose characters are completely explicit {(see \eqref{remark: character indexing})}.

The further step we have to take, having already computed the connected monodromy field $\Q(\varepsilon_{J_m})$, is to describe the action of the Galois group of $\Q(\varepsilon_{J_m})$ over $\Q$ on the Tate classes. 
Some of the necessary theory has been provided by Deligne in \cite{Deligne}: we expand upon it and 
unpack all the necessary tools to present an algorithm for the actual computation.
The reader is referred to \Cref{sec: full example ST15} for a thorough example, in which we compute the Sato-Tate group of $J_{m}$ for $m=15$; to our knowledge, this provides the first example in the literature of the Sato-Tate group of a degenerate abelian variety.

{Our method to compute the Sato-Tate group is flexible enough to handle several variants of the problem. In particular, we can handle the Jacobians of the hyperelliptic curves $C_{m,a} \colon y^2=x^m+a$ for arbitrary $a \in \Q^\times$ (see \Cref{subsect: twists}). The situation is especially clean for odd values of $m$, for which we prove that the Sato-Tate group of $\Jac(C_{m,a})$ is independent of $a$. More generally, though we don't work out the details, one could in principle handle arbitrary twists of $C_m$ (see \cite[Section 5.3]{part1}).}

We point out that our method of computing $\operatorname{ST}(J_m)$ requires us to go beyond Deligne's results in \cite{Deligne}, in that we need to describe very precisely the Galois action on the étale cohomology of (powers of) the Fermat curve $X_m^1 \colon x^m+y^m+z^m=0$. Deligne's results provide this description for the action of $\operatorname{Gal}(\overline{\Q}/\Q(\zeta_m))$, and one of our main contributions is to obtain results of the same precision for the action of the full Galois group $\operatorname{Gal}(\overline{\Q}/\Q)$.

This is especially hard because we want to pin down the Galois action on a \textit{specific} basis of cohomology. To elaborate on this point, we remark that it is comparatively easy to understand how the eigenspaces of $H^n(X^n_m)$ for the action of $G^n_m$ are permuted under the Galois action (see \Cref{lemma: Galois action and permutation of the characters}). Each of these eigenspaces is 1-dimensional, and we have written down a specific generator for each of them (essentially, the classes $v_f$ discussed above). Let $\omega_\alpha$ be a generator of the eigenspace indexed by the character $\alpha$.
For every Galois automorphism $\sigma \in \operatorname{Gal}(\overline{\Q}/\Q)$, we can easily compute the character $\beta$ such that $\sigma(\omega_{\alpha})$ is a scalar multiple of $\omega_\beta$. The real difficulty lies in understanding the proportionality factor $\omega_\beta = \lambda \cdot {\sigma(\omega_\alpha)}$.

Before stating our main theoretical result in this direction, we point out that $X_m^1$ carries an obvious action of the group $G_m^1 \colonequals \mu_m \times \mu_m \times \mu_m / \mu_m$ (with the three factors $\mu_m$ acting on the coordinates and the diagonal copy of $\mu_m$ acting trivially). This action extends to an action of $(G_m^1)^n$ on $(X_m^1)^n$, which leads naturally to the notion of a \textit{generalized eigenspace} in the cohomology of $(X_m^1)^n$, namely a vector subspace defined over $\Q$ that is stable under the action of $(G_m^1)^n$ (see the discussion surrounding \eqref{eq: decomposition of cohomology groups into generalised eigenspaces} for more details).
\begin{restatable*}{theorem}{TheoremGaloisActionTateClasses}\label{thm: explicit Galois action on Tate classes}
    Let $m \geq 3$ be an odd integer and $X^1_m/\Q$ be the smooth projective curve with equation $x^m+y^m+z^m=0$. For an integer $i$ with $1 \leq i \leq m-1$ let $\gamma_i$ be the character of $G_m^1$ given by $(i, i, -2i)$, see \Cref{remark: character indexing}.
    Let $\alpha = \gamma_{i_1} \ast \cdots \ast \gamma_{i_n} \in \mathfrak{B}_m^{3n-2}$ (see \Cref{eq: Amn and Bmn}) be a character such that the generalized eigenspace $\HH_{\operatorname{\acute{e}t}}^n((X_m^1)^n, \Q_\ell(\tfrac{n}{2}))_{[\alpha]}$ consists of Tate classes. Fix an embedding $\iota : \Q_\ell \hookrightarrow \C$ and consider the canonical isomorphism
    \[\HH_{\operatorname{\acute{e}t}}^n((X_m^1)^n, \Q_\ell(\tfrac{n}{2}))_{[\alpha]} \otimes_{\Q_\ell, \iota} \C \cong \HH^{n}_{\operatorname{dR}}((X_m^1)^n(\C), \C)(\tfrac{n}{2})_{[\alpha]}.
    \]
    We have a canonical basis $[\omega_\beta] \otimes 1$ of $\HH^{n}_{\operatorname{dR}}((X_m^1)^n, \C)(\tfrac{n}{2})_{[\alpha]}$ (see \Cref{eq: definition of omegai} and \Cref{eq: definition of omega alpha}), where $\beta$ ranges over the orbit $[\alpha] = \{u\alpha : u \in (\Z/m\Z)^\times\}$. With respect to this basis, the action of $\Gal(\overline{\Q}/\Q)$ on $\HH_{\operatorname{\acute{e}t}}^n((X_m^1)^n, \Q_\ell(\tfrac{n}{2}))_{[\alpha]} \otimes_{\Q_{\ell, \iota}} \C$ is described by the formula
    \[
        \tau([\omega_\beta]\otimes 1) = [\omega_{u(\tau)^{-1}\beta}] \otimes \frac{\mu_{u(\tau)^{-1}\beta} }{\mu_{\beta}} \cdot \frac{ \tau(P(\gamma, \omega_{-u(\tau)^{-1}\beta}))}{ P(\gamma,\omega_{-\beta})} \quad \forall \tau \in \Gal(\overline{\Q}/\Q),
    \]
    where 
    \begin{enumerate}
        \item $u(\tau) \in (\Z/m\Z)^\times$ is defined by the condition $\tau(\zeta_m)=\zeta_m^u$;
        \item     $\mu_\beta$ is the rational number defined in \Cref{lemma: Poincare isomorphism};
        \item %
    for  $\beta=\gamma_{i_1} \ast \cdots \ast \gamma_{i_n}$, the symbol $P(\gamma, \omega_\beta)$ denotes the algebraic number
    \[
    P(\gamma, \omega_\beta) = (2\pi i)^{-n/2} \prod_{r=1}^n \Gamma\left( \frac{i_r}{m} \right)^2 \Gamma\left( \frac{2i_r}{m} \right)^{-1}.
    \]
    \end{enumerate}

\end{restatable*}
\begin{remark*}
The rational numbers $\mu_\beta$ come from our choice of basis for the de Rham cohomology of $C_m \colon y^2=x^m+1$ (see \Cref{remark: the hyperelliptic ladic basis}), which is the traditional one for studying hyperelliptic curves. One could obtain slightly simpler formulas using a different basis (namely, rescaling $[\omega_\beta] \otimes 1$ by a factor of $\mu_\beta$), but doing so would make the result less easy to compare to the existing literature. See also \Cref{rmk: formulas without factor mu} for a more extended discussion.
\end{remark*}

\subsubsection*{The Gross-Koblitz formula} Work of Coleman \cite{Coleman} gives us another way to understand the Galois action on the cohomology of $X_m^1$ and its powers, in terms of values of Morita's $p$-adic Gamma function $\Gamma_p$. Comparing this description with  \Cref{thm: explicit Galois action on Tate classes} allows us to give a partial extension of a well-known formula of Gross and Koblitz \cite[Theorem 4.6]{GrossKoblitz} which relates the values of the usual $\Gamma$-function and of $\Gamma_p$ when evaluated at suitable rational arguments. This formula is restricted to primes congruent to $1$ modulo $m$ (equivalently, primes that split completely in $\Q(\zeta_m)$, or yet equivalently, primes whose Frobenius element lies in the absolute Galois group of $\Q(\zeta_m)$). Since \Cref{thm: explicit Galois action on Tate classes} describes the Frobenius action for \textit{all} rational primes, our formula does not have this restriction on $p$. On the other hand, since in \Cref{thm: explicit Galois action on Tate classes} we only work with certain subspaces of $\HH^1(X_m^1)^q$, we need to restrict the rational arguments at which we evaluate our various $\Gamma$ functions with respect to \cite[Theorem 4.6]{GrossKoblitz}. The precise statement is as follows:

\begin{restatable*}{theorem}{generalizedGK}
{\label{thm: Gross-Koblitz general version}}
Fix an odd integer $m \geq 3$ and let $\gamma_i$ be as in \Cref{thm: explicit Galois action on Tate classes}.
Let $\alpha = \gamma_{i_1} \ast \cdots \ast \gamma_{i_q}$ be a character such that $\langle u\alpha\rangle$ does not depend on $u \in (\Z/m\Z)^\times$, where $\langle u\alpha \rangle$ is defined in Equation \eqref{eq: Amn and Bmn}. Let $p$ be a prime not dividing $m$ and let $\operatorname{Frob}_p$ be a geometric Frobenius element at $p$, inducing the place $\mathcal{P}$ of $\Q(\varepsilon_{J_m})$ and hence the embedding $\iota_{\mathcal{P}} : \Q(\varepsilon_{J_m}) \hookrightarrow \Q(\varepsilon_{J_m})_{\mathcal{P}}$. We have
\[
\frac{ \hat{\Gamma}(-\alpha)}{\operatorname{Frob}_p\left( \hat{\Gamma}(-p\alpha)
\right)}
 = \iota_{\mathcal{P}}^{-1} \left(
 \frac{\hat{\Gamma}_p(p \alpha)}{
 (-1)^{\langle \alpha \rangle}}
 \right),
\]
where for $u \in (\Z/m\Z)^\times$ the quantities $\hat{\Gamma}(u\alpha)$ and $\hat{\Gamma}_p(u\alpha)$ are given by
\[
\hat{\Gamma}(u\alpha) = (2\pi i)^{-q/2} \prod_{j=1}^q \frac{\Gamma\left( \frac{[ui_j]}{m} \right)^2}{\Gamma\left( \frac{[2ui_j]}{m} \right)} \quad \text{ and } \quad \hat{\Gamma}_p(u \alpha) = \prod_{j=1}^q \frac{\Gamma_p\left( \frac{[ui_j]}{m} \right)^2}{\Gamma_p\left( \frac{[2ui_j]}{m} \right)},
\]
with $[a]$ being the unique integer representative of $a \in \Z/m\Z$ lying in the interval $[1, m]$.
\end{restatable*}

Apart from its intrinsic interest, this result allows us to test \Cref{thm: explicit Galois action on Tate classes} numerically. Indeed, the proof of \Cref{thm: Gross-Koblitz general version} shows that this formula is essentially equivalent to \Cref{thm: explicit Galois action on Tate classes}: we have tested the numerical equality of \Cref{thm: Gross-Koblitz general version} in various concrete cases and found that it always holds, which independently supports the result of \Cref{thm: explicit Galois action on Tate classes}.

It seems likely that the techniques in this paper could be extended to prove a version of \cite[Theorem 4.6]{GrossKoblitz} that holds for more general characters. In fact, to conclude this introduction we note that all the techniques we develop could easily be adapted to any quotient of the Fermat curve $X^1_m : x^m+y^m+z^m=0$, and we have focused on the curves $y^2=x^m+1$ mainly because these have already been studied extensively in the context of Sato-Tate groups. Although we won't go into the details here, let us briefly review the main ingredients of our approach and how they can be generalized. First of all, we need to compute equations for the Mumford-Tate group of $J_m$: we describe in \cite[Section 4]{part1} techniques that apply to any CM abelian variety. {This computation essentially amounts to determining generators for the Hodge ring of $J_m$, and in particular, it allows us to detect exceptional classes.} Second, our recipe for attaching Tate classes to equations of the Mumford-Tate group applies to any context where the Mumford-Tate group is an algebraic torus (all we need is that -- over a suitable extension -- the Mumford-Tate group can be realised as a group of diagonal matrices. This happens if and only if Mumford-Tate is an algebraic torus, which is equivalent to the abelian variety having complex multiplication). Finally, we can compute the Galois action on the Tate classes as soon as these can be embedded in the cohomology of a Fermat variety: a suitable variant of \cite[Lemma 5.2.3]{part1} impies the corresponding variant of \Cref{prop:ShiodaIII}, which yields this for the Jacobian of any quotient of $X^1_m$.

\subsection*{Acknowledgements} We thank Gregory Pearlstein for a useful discussion on Hodge theory. We are grateful to Johan Commelin, Francesc Fité, and Drew Sutherland for their comments on the first version of this manuscript.

H.G. was supported by NSF grant DMS-2201085. D.L.~was supported by the University of Pisa through grant PRA-2022-10 and by MUR grant PRIN-2022HPSNCR (funded by the European Union project Next Generation EU). D.L. and A.G.~are members of the INdAM group GNSAGA.

\section{Computation of the Sato-Tate group} {\label{section:Sato-Tate}}
We outline a strategy to compute the Sato-Tate group of the abelian variety $J_m/\Q$.
In this section, we start from the general version of the problem: given an abelian variety $A$ over a number field $K$, we describe the connected components of its monodromy group $\Gl\hookrightarrow \GL_{2g, \Q_\ell}$. Each connected component corresponds to an element $\tau$ of the Galois group $\Gal(\KconnA/K)$, and we identify a finite-dimensional tensor representation $W \subseteq T^\bullet (V_\ell \otimes V_\ell^\vee)$ such that the elements lying in the connected component corresponding to $\tau$ are precisely those which act on $W$ as $\tau$ does.

\subsection{Preliminaries}

Let $A$ be an abelian variety over a number field $K$. In this section we elaborate on Serre's construction \cite[\S 8.3]{MR2920749} of the Sato-Tate group of $A$ via $\ell$-adic étale cohomology. The construction depends in principle on an auxiliary prime $\ell$, but is known to be independent of this choice when $A$ satisfies the Mumford-Tate conjecture (see the main theorem of \cite{MR4530050}). We keep the discussion as general as possible, but we tailor it to our intended application to Fermat Jacobians (for which, in particular, the Mumford-Tate conjecture is known to hold).

\begin{definition}\label{def: Sato-Tate group} Let $A$ be an abelian variety defined over a number field $K$.
    Let $\ell$ be a prime and fix an arbitrary embedding $\iota : \Q_\ell \hookrightarrow \C$. Denote by $\Gl^1$ the intersection of $\Gl$ with $\operatorname{Sp}_{V_\ell, \phi_\ell}$, where $\phi_\ell$ is the bilinear form induced on $V_\ell$ by a fixed polarisation $\phi$ of $A$.
    The Sato-Tate group $\ST_A$ of $A/K$ is a maximal compact subgroup of $(\Gl^1 \times_{\iota} \C) (\C)$, well-defined up to conjugacy in $(\Gl^1 \times_{\iota} \C) (\C)$.
\end{definition}

It is known \cite[\S 8.3.4]{MR2920749} that $\Gl$ and $\ST_A$ have isomorphic component groups. More precisely, the canonical map $\ST_A / \ST_A^0 \to \Gl/\Gl^0$ is an isomorphism.

\begin{remark}
    In the case of Fermat Jacobians, the identity component $\Gl^0$ is a torus, say of rank $r$. Since $\C$ is algebraically closed, we obtain that the identity component of $(\Gl^1 \times_\iota \C)(\C)$ is isomorphic to $\mathbb{G}_m(\C)^{r-1} \cong (\C^\times)^{r-1}$, and it follows that the identity component of $\ST_A$ is isomorphic to $U(1)^{r-1} = (\mathbb{S}^1)^{r-1} \subset (\C^\times)^{r-1}$.
\end{remark}

We fix once and for all an auxiliary prime $\ell$. When $A$ is the Fermat Jacobian $J_m$, we take $\ell$ to be congruent to $1$ modulo $m$, and also fix an embedding of $\Q(\zeta_m)$ inside $\Q_\ell$.
Our description of $\Gl$ (hence of $\ST_A$) relies on the study of Tate classes on powers of $A$. Formally, we introduce the following notation:
    \begin{definition}
            Let $V=\HH^1(A_{\C}(\C), \Q)$ and $V_\ell = \HH^1_{\operatorname{\acute{e}t}}( A_{\overline{K}}, \mathbb{Q}_\ell)\cong V \otimes \mathbb{Q}_\ell$. There is a natural action of $\mathcal{G}_\ell$ on $V_\ell$ and on $V_\ell^\vee$. We let
    \[
    W_n \colonequals ( (V_\ell \otimes V_\ell^\vee)^{\otimes n})^{\mathcal{G}_\ell^0}
    \]
    and
    \[
    W \colonequals \bigoplus_{n \geq 0} W_n.
    \]
    \end{definition}

\begin{remark}
    If $A$ satisfies the Mumford-Tate conjecture, the space $W_n$ can be identified with $\left((V \otimes V^\vee)^{\otimes n}\right)^{\operatorname{MT}(A)} \otimes \Q_\ell$, that is, $W_n$ is the base-change from $\Q$ to $\Q_\ell$ of the space of Hodge classes in the Hodge structure $V^{\otimes n} \otimes (V^\vee)^{\otimes n}$. We observe that this structure is pure of weight $0$.
\end{remark}
    
    It is clear that $W$ is a subspace (and in fact a sub-algebra) of the tensor algebra 
    \[
    T^\bullet (V_\ell \otimes V_\ell^\vee) \colonequals \bigoplus_{n \geq 0} (V_\ell \otimes V_\ell^\vee)^{\otimes n}
    \]
    over $V_\ell \otimes V_\ell^\vee$. More precisely, it is the largest sub-algebra (equivalently, vector subspace) of this tensor algebra on which $\mathcal{G}_\ell^0$ acts trivially.

    \begin{definition}
    The whole group $\GL_{V_\ell}$ acts naturally on the tensor algebra $T^\bullet (V_\ell \otimes V_\ell^\vee)$. We denote by $\rho_{T^\bullet (V_\ell \otimes V_\ell^\vee)}$ the natural representation $\GL_{V_\ell} \to \GL_{T^\bullet (V_\ell \otimes V_\ell^\vee)}$.
    \end{definition}

    \begin{proposition}\label{prop: connected l-adic monodromy group is pointwise stabiliser of W}
    The maximal (algebraic) subgroup of $\GL_{V_\ell}$ that acts trivially on $W$ in the representation $\rho_{T^\bullet(V_\ell \otimes V_\ell^\vee)}$ is $\Gl^0$. %
    \end{proposition}
        \begin{proof}
    This is essentially a consequence of Chevalley's theorem. By \cite[Proposition 3.1(c)]{Deligne}, the reductive group $\mathcal{G}_\ell^0$ is precisely the subgroup of $\GL_{V_\ell}$ that acts trivially on
    \[
    Z := \{ v \in \bigoplus_{m, n \geq 0} V_\ell^{\otimes m} \otimes (V_\ell^\vee)^{\otimes n} \bigm\vert \mathcal{G}_\ell^0 \text{ acts trivially on } v \}.
    \]
    A theorem of Bogomolov \cite{MR0587337} shows that $\Gl^0$ contains the torus of homotheties. The action of the homothety $\lambda \operatorname{Id}$ on $V_\ell^{\otimes m} \otimes (V_\ell^\vee)^{\otimes n}$ is multiplication by $\lambda^{m-n}$, so $V_\ell^{\otimes m} \otimes (V_\ell^\vee)^{\otimes n}$ can contain vectors fixed by $\mathcal{G}_\ell^0$ (equivalently, Tate classes) only when  $m=n$. It follows that $\mathcal{G}_\ell^0$ is the stabiliser of the subspace $Z \cap \bigoplus_{n \geq 0} (V_\ell \otimes V_\ell^\vee)^{\otimes n}$, which by definition is $W$.
    \end{proof}

The algebra $W$ of all Tate classes is too large for our purposes of describing the group $\Gl$ in finite terms. For this reason, we introduce the following truncated version of $W$:
\begin{definition}[Space of Tate classes of degree up to $r$]
\label{definition: Space of Tate classes of degree up to r W}
    For every positive integer $r$ we set
    \[
    W_{\leq r} \colonequals \bigoplus_{n \leq r} W_n.
    \]
\end{definition}

Our next result is the crucial ingredient to show that, in order to study the group $\Gl$, it suffices to consider $W_{\leq N}$ for $N$ sufficiently large.

\begin{corollary}\label{cor: algebra of Tate classes is finitely generated}
There exists $N \geq 1$ such that the group $\Gl^0$ is the subgroup of $\GL_{V_\ell}$ that acts trivially on $W_{\leq N}$.
\end{corollary}
\begin{proof}
    By the proof of \Cref{prop: connected l-adic monodromy group is pointwise stabiliser of W}, the group $\Gl^0$ is defined inside $\GL_{V_\ell}$ by the ideal $I \colonequals \langle g \cdot w = w \mid w \in W \rangle$. Since the coordinate ring of $\operatorname{GL}_{V_\ell}$ is Noetherian, there exist finitely many $w_1, \ldots, w_k \in W$ such that $I=\langle g \cdot w_i = w_i \mid i=1,\ldots,k\rangle$. Each of these $w_i$ lies in $W_{n(i)}$ for some $n(i)$, and it suffices to take $N=\max_i n(i)$.
\end{proof}
\begin{remark}\label{rmk: can look at a smaller space of Tate classes than the truncation}
    By the same argument, one can in fact replace $W_{\leq N}$ with the smallest subspace $\tilde{W}$ of $W$ that is stable under $\Gl$ and contains $w_1, \ldots, w_k$ (with notation as in the proof of \Cref{cor: algebra of Tate classes is finitely generated}).
    \end{remark}

\begin{corollary}\label{cor: replace W by its truncation}
    Let $N$ be as in \Cref{cor: algebra of Tate classes is finitely generated}. A subgroup of $\GL_{V_\ell}$ acts trivially on $W$ if and only if it acts trivially on $W_{\leq N}$.
\end{corollary}
\begin{proof}
    One implication is trivial. For the other, suppose that $H \subseteq \GL_{V_\ell}$ acts trivially on $W_{\leq N}$. By \Cref{cor: algebra of Tate classes is finitely generated} we have $H \subseteq \Gl^0$, and therefore $H$ acts trivially on $W$, because $\Gl^0$ does.
\end{proof}

By construction, $W$ and $W_{\leq N}$ are trivial representations of the group $\Gl^0$. Our next results will allow us to consider them as representations of the larger group $\Gl$.
    
    \begin{lemma}\label{lemma: action of Gl on W}
        The map $\rho_{T^\bullet}$ endows $W$ with the structure of a representation of the group $\Gl$. The restriction of this representation to $\Gl^0$ is trivial. The same holds with $W$ replaced by $W_{\leq r}$, where $r$ is any positive integer.
    \end{lemma}
    \begin{proof}
        The fact that $\Gl^0$ acts trivially on $W$ is immediate from the definition of this space. For the first part of the lemma, we need to check that $\Gl$ sends $W$ to itself. Functorially, we need to check that for any $\Q_\ell$-algebra $R$ and any $h \in \Gl(R)$ the subspace $W \otimes_{\Q_\ell}R$ is stable under $h$. To see this, it suffices to check that $h(W \otimes_{\Q_\ell}R) \subseteq W \otimes_{\Q_\ell} R$. Since $W \otimes_{\Q_\ell} R$ is the space of fixed points of $\Gl^0(R)$ acting on $T^\bullet(V_\ell \otimes V_\ell^\vee) \otimes_{\Q_\ell} R$, it further suffices to check that for all $h' \in \Gl^0(R)$ the action of $h'$ on $h(W \otimes_{\Q_\ell}R)$ is trivial. Since $\Gl^0$ is normal in $\Gl$, there exists $h'' \in \Gl^0(R)$ such that
        \[
        h' h = h h''.
        \]
        It follows that, for all $w \otimes r \in W \otimes R$, we have 
        \[
        h' h (w \otimes r) = h h'' (w \otimes r) = h(w \otimes r),
        \]
        where we have used the fact that $h'' \in \Gl^0(R)$ acts trivially on $W \otimes R$. This shows as desired that $h'$ acts trivially on $h(W \otimes_{\Q_\ell}R)$ and concludes the proof. Finally, it is easy to check that the same argument also applies to $W_{\leq r}$.
    \end{proof}

\Cref{lemma: action of Gl on W} allows us to define a representation of $\Gl/\Gl^0$ on $W$ (or on $W_{\leq r}$). Indeed, there is an action of $\Gl$ on $W$ via $\rho_{T^\bullet(V_\ell \otimes V_\ell^\vee)}$. This action is trivial on $\Gl^0$, hence induces a representation
    \[
    \Gl/\Gl^0 \to \GL_W.
    \]
    \Cref{prop: connected l-adic monodromy group is pointwise stabiliser of W} implies that this representation is faithful.
The definition of $\mathcal{G}_\ell$ as the Zariski closure of $\rho_{A, \ell}(\Gal(\overline{K}/K))$ shows that every connected component of $\mathcal{G}_\ell$ contains an element of the form $\rho_{A, \ell}(\sigma)$ for some $\sigma \in \Gal(\overline{K}/K)$. Since by definition $\sigma \in \Gal(\overline{K}/K)$ maps to $\mathcal{G}_\ell^0(\Q_\ell)$ if and only if $\sigma$ is in $\Gal(\overline{K}/\KconnA)$, we see that the connected component of $\Gl$ in which $\rho_{A, \ell}(\sigma)$ lies depends only on the image of $\sigma$ in
\[
\frac{\Gal(\overline{K}/K)}{\Gal(\overline{K}/\KconnA)} = \Gal(\KconnA / K).
\]
In turn, this implies that
\[
\Gal(\overline{K}/K) \xrightarrow{\rho_{A, \ell}} \mathcal{G}_\ell(\Q_\ell) \xrightarrow{\rho_W} \GL_W(\Q_\ell)
\]
factors via $\Gal(\KconnA/K)$, giving a faithful representation
\begin{equation}\label{eq: representation of the Galois group of Kconn in W}
    \rho_W : \Gal(\KconnA/K) \to \GL_W(\Q_\ell).
\end{equation}
As before, since an element acts trivially on $W_{\leq N}$ if and only if it acts trivially on $W$, we can replace $W$ with $W_{\leq N}$ in the above discussion. We denote by $\rho_{W_{\leq N}}$ the representation thus obtained.
For all integers $r \geq 1$ we also have a representation
\begin{equation}\label{eq: representation of the Galois group of Kconn in W truncated at r}
    \rho_{W_{\leq r}} : \Gal(\KconnA/K) \to \GL_{W_{\leq r}}(\Q_\ell)
\end{equation}
which however need not be faithful if $r$ is too small.

\subsection{Description of the group \texorpdfstring{$\Gl$}{Gl}}

    We now show how to describe the full group $\mathcal{G}_\ell$ (and not just the identity component $\mathcal{G}_\ell^0$) in terms of its action on $W$.
    \begin{proposition}
        Fix $h_0 \in \mathcal{G}_\ell(\Q_\ell)$. Let $h \in \GL_{V_\ell}(\Q_\ell)$ satisfy the following:
        \begin{enumerate}
            \item Consider the action $\rho_{T^\bullet(V_\ell \otimes V_\ell^\vee)}$ of $\GL_{V_\ell}$ on $T^\bullet (V_\ell \otimes V_\ell^\vee)$. We assume that via this action we have $\rho_{T^\bullet (V_\ell \otimes V_\ell^\vee)}(h)(W) \subseteq W$ and denote by $h|_W$ the action of $h$ on $W$;
            \item $h|_W = h_0|_W$.
        \end{enumerate}
Then $h$ lies in $\mathcal{G}_\ell(\Q_\ell)$, and more precisely, it is in the same irreducible component of $\Gl$ as $h_0$. The same statement holds with $W$ replaced by $W_{\leq N}$.
    \end{proposition}
\begin{proof}
    Since $\mathcal{G}_\ell(\Q_\ell)$ is a group, $h$ is in $\mathcal{G}_\ell(\Q_\ell)$ if and only if $hh_0^{-1}$ is. Now, $hh_0^{-1}|_W=h|_W h_0|_W^{-1}$, which by assumption is the identity. This means that $hh_0^{-1}$ acts trivially on $W$, hence by Proposition \ref{prop: connected l-adic monodromy group is pointwise stabiliser of W} we have $hh_0^{-1} \in \mathcal{G}_\ell^0(\Q_\ell)$. %
    This concludes the proof for the case of $W$. The case of $W_{\leq N}$ is completely analogous, using the fact that the action of an element on $W_N$ is trivial if and only if it is trivial on all of $W$ (\Cref{cor: replace W by its truncation}).
\end{proof}

Thus, the group $\mathcal{G}_\ell(\Q_\ell)$ can be described as follows. Let $g_1, \ldots, g_r \in \mathcal{G}_\ell(\Q_\ell)$ be elements lying in the $r$ connected (equivalently, irreducible) components of $\mathcal{G}_\ell$. Then, omitting for simplicity the representation $\rho_{T^\bullet}$, we have
\[
\mathcal{G}_\ell(\Q_\ell) = \left\{ g \in \GL_{V_\ell}(\Q_\ell) \bigm\vert \begin{array}{c}
     g(W) \subseteq W \\
     \exists i : g|_W = g_i|_W
\end{array} \right\} = \bigsqcup_{i=1}^r \left\{ g \in \GL_{V_\ell}(\Q_\ell) \bigm\vert \begin{array}{c}
     g(W) \subseteq W \\
     g|_W = g_i|_W
\end{array} \right\}.
\]
We have already observed that every connected component of $\Gl$ contains an element $h$ of the form $h=\rho_{A, \ell}(\sigma)$ for some $\sigma \in \Gal(\overline{K}/K)$. We also know (see \Cref{eq: representation of the Galois group of Kconn in W}) that $h|_{W_{\leq N}}$ only depends on the image of $\sigma$ in $\Gal(\KconnA/K)$.
We can then summarise the above discussion in the following theorem:
\DescriptionOfGl

It will be useful to consider `truncated' versions of the group appearing in the previous theorem, where we consider the action on $W_{\leq r}$ for arbitrary $r$.

\begin{lemma}
    Let $r$ be a positive integer.
    \begin{enumerate}
        \item The set
        \begin{equation}\label{eq: group H leq r}
            H_{\leq r} \colonequals \bigsqcup_{\sigma \in \Gal(\KconnA/K)} \left\{ h \in \GL_{V_\ell}(\Q_\ell) \bigm\vert \begin{array}{c}
h(W_{\leq r}) \subseteq W_{\leq r} \\
h|_{W_{\leq r}} = \rho_{W_{\leq r}}(\sigma)
\end{array}
\right\}
        \end{equation}
is a group.
\item We have $\Gl(\Q_\ell) \subseteq H_{\leq r}$.
\item If $r$ is sufficiently large (in particular, if $r \geq N$, where $N$ is as in \Cref{cor: algebra of Tate classes is finitely generated}), then $H_{\leq r} = \Gl(\Q_\ell)$.
    \end{enumerate}
\end{lemma}
\begin{proof}
    Part (1) is a straightforward verification and part (3) follows from \Cref{thm: description of Gl in terms of Tate classes}. Part (2) is simply the statement that every element in $\Gl(\Q_\ell)$ sends $W_{\leq r}$ to itself (see \Cref{lemma: action of Gl on W}) and that it acts on it as an element of the Galois group $\Gal(\KconnA/K)$ (see \Cref{eq: representation of the Galois group of Kconn in W truncated at r}).
\end{proof}

Note that when $\sigma$ is the identity element of the group $\Gal(\KconnA/K)$, the set
\begin{equation}\label{eq: identity component in terms of Tate classes, truncated version}
H_{\leq r}^0 \colonequals \left\{ h \in \GL_{V_\ell}(\Q_\ell) \bigm\vert \begin{array}{c}
h(W_{\leq r}) \subseteq W_{\leq r} \\
h|_{W_{\leq r}} = \rho_{W_{\leq r}}(\sigma)
\end{array}
\right\} = \left\{ h \in \GL_{V_\ell}(\Q_\ell) \bigm\vert 
h|_{W_{\leq r}} = \operatorname{id}_{W_{\leq r}}
\right\}
\end{equation}
is a subgroup of $H_{\leq r}$ that, by our previous considerations (see in particular \Cref{cor: algebra of Tate classes is finitely generated}), contains $\Gl^0(\Q_\ell)$.

\subsection{The action on \texorpdfstring{$W_n$}{Wn} for small \texorpdfstring{$n$}{n}}
In this section we collect some remarks about the information on $\Gl$ that can be gleaned from its action on $W_n$ for small values of $n$ (mostly $n=1$ and $2$). %

\begin{remark}[$n=1$: endomorphisms]\label{rmk: n=1 and endomorphisms}
    Taking $n=1$, that is, looking at $(V_\ell \otimes V_\ell^\vee)^{\mathcal{G}_\ell^0}$, is the same as considering endomorphisms. Indeed, by Faltings's theorem one has
    \begin{equation}\label{eq: subspace W_1 and endomorphisms}
    (V_\ell \otimes V_\ell^\vee)^{\mathcal{G}_\ell^0} = \operatorname{End}(V_\ell)^{\mathcal{G}_\ell^0} = \operatorname{End}_{{\mathcal{G}_\ell^0}}(V_\ell) = \operatorname{End}(A_{\overline{K}}) \otimes \mathbb{Q}_\ell.
    \end{equation}
The natural action of $\GL_{V_\ell}(\Q_\ell)$ on $V_\ell \otimes V_\ell^\vee = \operatorname{End}(V_\ell)$ is conjugation: an element $h \in \GL_{V_\ell}$ acts on $e \in \operatorname{End}(V_\ell)$ as
\[
h \cdot e = h \circ e \circ h^{-1}.
\]
The Galois action $\rho_W$, when restricted to $W_1$, is simply the natural Galois action on the endomorphism ring $\operatorname{End}(A_{\overline{K}})$. Thus, the condition for $h$ to lie in $H_{\leq 1}$ is that, for every $e \in \operatorname{End}(A_{\overline{K}})$, we have
\begin{equation}\label{eq: twisted Lefschetz condition}
h \circ e \circ h^{-1} =\rho_{W_{\leq 1}}(\sigma)(e) = \sigma(e)
\end{equation}
for some $\sigma \in \Gal(\KconnA/K)$. We also note that $\rho_{W_{\leq 1}}$, being the natural Galois action on the geometric endomorphism ring of $A$, factors via the quotient map $\Gal(\KconnA/K) \to \Gal(K(\End(A))/K)$.
This equation should be compared to \cite[Equation (5.2)]{MR3320526} and \cite[Equation (3.2)]{MR3502937}. Note however that -- unlike the two works just cited -- we are not asking that $h$ preserves the polarisation: this condition will come naturally from considering the action on $W_2$, as we explain in the next remark.
\end{remark}

\begin{remark}[$n=2$ and the class of the polarisation]\label{rmk: n=2 and the polarisation}
    We describe how a polarisation on $A$ gives rise to a class in $W_2$. Let $\varphi : V \to V^\vee$ be a fixed polarisation. It can equivalently be considered as an element of $V^\vee \otimes V^\vee$. We show that $\varphi$ induces a class in $W_2$ that we describe explicitly. Let $w_1, \ldots, w_{2g}$ be a basis of $V_\ell$ and let $w_1^*, \ldots, w_{2g}^*$ be the basis of $V_\ell$ which is dual to $w_1, \ldots, w_{2g}$ with respect to $\varphi \otimes \Q_\ell$. A straightforward computation shows that the element $C := \sum_{i=1}^{2g} w_i \otimes w_i^* \in V_\ell^{\otimes 2}$ is independent of the choice of the basis $\{w_i\}$. %
    Now consider the element $C \otimes \varphi \in V_\ell \otimes V_\ell \otimes V_\ell^\vee \otimes V_\ell^\vee$, and suppose that an element $h \in \GL_{V_\ell}(\Q_\ell)$ fixes it.
    The fact that $h \cdot (C \otimes \varphi) = C \otimes \varphi$ implies %
    \[
    (h^{\otimes 2}C) \otimes \varphi(h^{-1} \cdot, h^{-1} \cdot) = C \otimes \varphi(\cdot, \cdot),
    \]
    which in turn gives $\varphi(h^{-1} \cdot, h^{-1} \cdot) = \mu^{-1} \varphi(\cdot, \cdot)$ and ${h^{\otimes 2} C}=\mu C$ for some scalar $\mu$. In particular, $h^{-1}$ is a general symplectic transformation with multiplier $\mu^{-1}$ (hence $h$ is symplectic with multiplier $\mu$). Thus, if $h$ stabilises the class $C \otimes \varphi$, then it lies in $\GSp_{V_\ell, \varphi}(\Q_\ell)$.
    
    Conversely, to see that the class $C \otimes \varphi$ really is invariant under $\GSp_{V_\ell,\varphi}$, note that -- given $h$ in $\GSp_{V_\ell,\varphi}(\Q_\ell)$ with multiplier $\lambda$ -- one has $\varphi(h^{-1} \cdot, h^{-1} \cdot) = \lambda^{-1}\varphi(\cdot, \cdot)$, and on the other hand $h$ sends the bases $w_1, \ldots, w_{2g}$, $w_1^*, \ldots, w_{2g}^*$ to $u_1=hw_1, \ldots, u_{2g}=hw_{2g}$ and $hw_1^*, \ldots, h w_{2g}^{*}$. For this pair of bases, one has $\varphi(u_i, hw_k^{*}) = \varphi(h w_i, hw_k^*) = \lambda \delta_{ik}$. Hence, $u_i^* = \frac{1}{\lambda} hw_i^*$, and in particular
    \[
    h^{\otimes 2}C = \sum_i (hw_i) \otimes (hw_i^*) = \sum_i u_i \otimes \lambda u_i^* = \lambda \sum_i u_i \otimes u_i^* = \lambda C,
    \]
    and therefore $C \otimes \varphi$ is indeed invariant (the first factor gets rescaled by $\lambda$, the second factor by $\lambda^{-1}$). Note that -- since the polarisation is defined over the ground field -- the whole group $\Gl$ stabilises the class $C \otimes \varphi$ just defined.
\end{remark}

The upshot of \Cref{rmk: n=1 and endomorphisms,rmk: n=2 and the polarisation} is that looking at $W_{\leq 2}$ is guaranteed to take into account both the polarisation and the compatibility with endomorphisms. However, considering the action on $W_{\leq 2}$ may capture more information than this: see \Cref{sec: full example ST15}, where we found some relations of degree 2 that do not correspond to the polarization.

\begin{remark}[Relation with the twisted Lefschetz group]\label{rmk: relation with the twisted Lefschetz group}
    The twisted (decomposable) Lefschetz group $\operatorname{TL}_A$ of Banaszak and Kedlaya \cite[Definition 3.4]{MR3502937} is the subgroup of $\GL_{V}$ given by those automorphisms that preserve the polarisation up to scalars and commute with the endomorphisms of $A$ up to the Galois action, in the sense of Equation \eqref{eq: twisted Lefschetz condition}. \Cref{rmk: n=2 and the polarisation,rmk: n=1 and endomorphisms} show that $H_{\leq 2}$ is contained in the $\Q_\ell$-points of the twisted Lefschetz group, because elements of $H_{\leq 2}$ preserve the polarisation up to scalars and commute with the geometric endomorphisms of $A$, up to the Galois action. 
    
    Suppose furthermore that the following all hold:
    \begin{enumerate}
        \item the abelian variety $A$ satisfies the Mumford-Tate conjecture,
        \item its Hodge group coincides with the identity component of the Lefschetz group;
        \item the twisted Lefschetz group of $A_{\overline{K}}$ is connected.
    \end{enumerate} 
    Abelian varieties satisfying (1) and (2) are well-studied and sometimes called \textit{fully of Lefschetz type}, see for example \cite[Definition 2.4]{banaszak2022remark} or \cite[Definition 2.22]{MR3494170}.
    If hypotheses (1)-(3) are satisfied, we claim that taking $N=2$ in our construction we have 
    \[
    H_{\leq 2} = \operatorname{TL}_A(\Q_\ell)= \Gl(\Q_\ell),
    \]
    so that we recover exactly the twisted Lefschetz group. To see that the above equalities hold, notice that assumptions (1)-(3) are precisely the hypotheses of  \cite[Theorem 6.1]{MR3320526}, which guarantees that the equality $\operatorname{TL}_A(\Q_\ell)= \Gl(\Q_\ell)$ holds.
    On the other hand, we have already noticed that $H_{\leq 2}$ always contains $\Gl(\Q_\ell)$ and is always contained in $\operatorname{TL}_A(\Q_\ell)$. We thus obtain the chain of inclusions
    \[
    \operatorname{TL}_A(\Q_\ell) = \Gl(\Q_\ell)  \subseteq H_{\leq 2} \subseteq \operatorname{TL}_A(\Q_\ell),
    \]
    which proves the desired equalities.
\end{remark}

\subsection{The case of Fermat Jacobians}\label{sect: Sato-Tate for Fermat Jacobians}
In this section, we specialize our previous remarks on the structure of $\Gl$ to the case of a Fermat Jacobian $J_m$. As in previous sections, we will relate the cohomology of $J_m$ to that of the Fermat hypersurfaces $X_m^n$ and of $(X_m^1)^n$. Recall from \Cref{remark: character indexing} the group $G_m^n$, which will be important to consider as an algebraic group over $\Q$. With this convention, the action of $G_m^n$ on $X_m^n$ is defined over $\Q$.

Recall that we have fixed our auxiliary prime $\ell$ to be totally split in $\Q(\zeta_m)$ and we have fixed an embedding $\Q(\zeta_m) \hookrightarrow \Q_\ell$. In particular, this allows us to consider the eigenspaces for the action of the automorphism $\alpha_m$ of $C_m$ also in $\ell$-adic étale cohomology. For any character $\alpha$, the corresponding eigenspace is 1-dimensional.
We pick a basis of $\HH^1_{\text{ét}}((J_m)_{\overline{\Q}}, \Q_\ell)$ consisting of eigenvectors for the action of the automorphism $\alpha_m$ of $C_m$ (see \Cref{remark: the hyperelliptic ladic basis}). We denote this basis by $v_1, \ldots, v_{m}$, where $\alpha_m^* v_i = \zeta_m^{i} v_i$, for all $1 \leq i \leq m$, $i \neq \frac{m}{2}$.
We will also make use of the basis $\omega_i$ of $\HH^1_{\operatorname{dR}}(J_m, \C)$ defined in \Cref{remark: the hyperelliptic ladic basis}.
Recall that $\omega_i$ satisfies $\alpha_m^* \omega_i=\zeta_m^i \omega_i$. 
We note that, given an embedding $\iota : \Q_\ell \hookrightarrow \C$, we have isomorphisms
\[
\HH_{\text{ét}}^1((J_m)_{\overline{\Q}}, \Q_\ell) \otimes_{\Q_\ell, \iota} \C \cong \HH_{B}^1(J_m(\C), \Q) \otimes_{\Q} \Q_\ell \otimes_{\Q_\ell, \iota} \C \cong \HH_{\operatorname{dR}}^1(J_m, \C).
\]
Via this identification, $v_i \otimes 1$ and $\omega_i$ both lie in the same $1$-dimensional eigenspace of $\HH_{\operatorname{dR}}^1(J_m, \C)$, hence these two vectors are proportional. %

If $\{v_i^\vee\}$, respectively $\{\omega_i^\vee\}$, denotes the basis of $V_\ell^\vee$ dual to $v_i$ (resp.~the basis of $(V_\ell \otimes_{\Q_\ell, \iota} \C)^\vee \cong \HH_{\operatorname{dR}}^1(J_m, \C)^\vee$ dual to $\omega_i$) we have $(v_i \otimes 1) \otimes (v_i^\vee \otimes 1) = \omega_i \otimes \omega_i^\vee$. For simplicity of notation, from now on we will sometimes consider $V_\ell$ as a subspace of $V_\ell \otimes_{\Q_\ell, \iota} \C$ and write $v_i \otimes 1$ simply as $v_i$.

Use the basis $\{v_i\}$ to identify $\GL_{V_\ell}$ to $\GL_{2g, \Q_\ell}$. With this identification, \cite[Lemma 4.3.4]{part1} and its counterpart \cite[Lemma 4.2.3]{part1} both tell us that $\Gl^0$ is a subgroup of the diagonal torus of $\GL_{2g, \Q_\ell}$.
From now on, we will use this identification. In particular, we will speak of the \textit{diagonal torus} of $\GL_{V_\ell}$, meaning the maximal torus of $\GL_{V_\ell}$ that corresponds to the torus of diagonal matrices in $\GL_{2g, \Q_\ell}$ under the identification provided by the basis $\{v_i\}$.

\cite[Lemma 4.3.4]{part1} implies in particular that $\Gl^0$ is a split algebraic torus, and any representation of a split algebraic torus is the direct sum of its eigenspaces. For every $n \geq 1$, any class of the form $v_{i_1} \otimes v_{i_2} \otimes \cdots \otimes v_{i_n} \otimes v_{j_1}^\vee \otimes \cdots \otimes v_{j_n}^\vee$ generates such an eigenspace, and the sum of these eigenspaces spans all of $V_\ell^{\otimes n} \otimes (V_\ell^{\vee})^{\otimes n}$. It follows that (for every $n \geq 1$) the subspace $W_n$ is the direct sum of some of these eigenspaces, and we can limit ourselves to considering such decomposable tensors. (This corresponds to the fact that $\Gl^0$ can be defined inside the diagonal torus of $\GL_{2g, \Q_\ell}$  by monomial equations of the form $\prod_{h=1}^n x_{i_h} = \prod_{h=1}^n x_{j_h}$).

With these preliminaries taken care of, we are now able to give an admissible value for the integer $N$ of \Cref{cor: algebra of Tate classes is finitely generated}. 
\begin{lemma}\label{lemma: degree of equations is the same as degree of classes}
Let $A=J_m$ be a Fermat Jacobian. Let $f_1, \ldots, f_r$ be a set of defining equations for $\MT(A)$, where each $f_i$ is a monomial $\prod_{j} x_j^{d_{i, j}}$. Let $q_i=\sum_{d_{i,j} \geq 0} d_{i,j}$ be the sum of the positive exponents of $f_i$, and let $q=\max q_i$. We have $q \geq 2$, and in \Cref{cor: algebra of Tate classes is finitely generated} we can take $N=q$.
\end{lemma}
\begin{proof}
We will use repeatedly the fact that, since the Mumford-Tate conjecture holds for the CM abelian varieties $J_m$, we have a canonical isomorphism $\MT(A)_{\Q_\ell} \cong \Gl^0$.

Denote by $q_i^-=\sum_{d_{i,j} \leq 0} |d_{i,j}|$ the absolute value of the sum of the negative exponents in $f_i$. Because the group of homoteties is contained in the Mumford-Tate group, one has $q_i=q_i^-$ \cite[Lemma 4.2.10]{part1}.
For every equation $f_i$, consider the class $c_i \colonequals \bigotimes_{j} v_j^{\otimes d_{i,j}}$, where -- for a negative exponent $d_{i,j}$ -- the tensor power $v_j^{\otimes d_{i,j}}$ denotes
\[
\underbrace{v_j^\vee \otimes \cdots \otimes v_j^\vee}_{|d_{i,j}| \text{ times}} \in (V_\ell^\vee)^{\otimes |d_{i,j}|}.
\]
Up to reordering the tensor factors, the class $c_i$ belongs to $V_\ell^{\otimes q_i} \otimes (V_\ell^\vee)^{\otimes q_i^{-}} = (V_\ell \otimes V_\ell^\vee)^{\otimes q_i}$. As in \cite[Section 4.3]{part1}, the action of the diagonal matrix $x = \operatorname{diag}(x_1, \ldots, x_{2g})$ on $c_i$ is given by multiplication by $f_i(x_1, \ldots, x_{2g})$. Hence, $c_i$ is fixed by $x$ if and only if $f_i(x)=1$, and therefore, an element $x$ in the diagonal torus of $\GL_{V_\ell}$ fixes all classes $c_i$ if and only if it satisfies $f_i(x)=1$ for all $i$, if and only if $x$ belongs to $\MT(A)(\Q_\ell)=\Gl^0(\Q_\ell)$.
Moreover, \Cref{rmk: n=2 and the polarisation} shows that $q \geq 2$, and \Cref{rmk: n=1 and endomorphisms} shows that a matrix $x \in \GL_{V_\ell}(\Q_\ell)$ that acts trivially on $W_1$ is necessarily diagonal. 

To conclude the proof, it suffices to show that the group $H^0_{\leq q}$ (see \eqref{eq: identity component in terms of Tate classes, truncated version}) coincides with $\Gl^0(\Q_\ell)$. We show the double containment:
\begin{itemize}
    \item On the one hand, the subgroup $H^0_{\leq q}$ of $\GL_{V_\ell}(\Q_\ell)$ that acts trivially on $\bigoplus_{n \leq q} W_n$ is contained in the diagonal torus (see also \Cref{rmk: n=1 and endomorphisms for Fermat Jacobians} for more details) and satisfies all the equations that define $\MT(A)$, see above. Thus, $H^0_{\leq q} \subseteq \MT(A)(\Q_\ell)=\Gl^0(\Q_\ell)$.
    \item On the other hand, $\MT(A)(\Q_\ell) = \Gl^0(\Q_\ell)$ acts trivially on all the $W_n$ by construction, hence $\MT(A)(\Q_\ell)=\Gl^0(\Q_\ell)$ is contained in $H^0_{\leq q}$.
\end{itemize}
\end{proof}

Our next two observations are the special cases of \Cref{rmk: n=1 and endomorphisms,rmk: n=2 and the polarisation} for Fermat Jacobians.

\begin{remark}[$n=1$: endomorphisms of Fermat Jacobians]\label{rmk: n=1 and endomorphisms for Fermat Jacobians}
    The non-trivial endomorphisms of $J_m$ given by the action of the $m$-th roots of unity are represented (with respect to the basis $\{\omega_i\}$) by the classes
    \[
    \psi_k \colonequals \sum_i \zeta_m^{ki} \omega_i \otimes \omega_i^\vee = \sum_i \zeta_m^{ki} v_i \otimes v_i^\vee \quad \text{for }k = 1, \ldots, m-1,
    \]
    where the last equality follows from our previous observation that $v_i \otimes v_i^\vee= \omega_i \otimes \omega_i^\vee$.
    Since $\mu_m \subset \Q_\ell^\times$, a simple argument using Vandermonde determinants shows that all the classes $v_i \otimes v_i^\vee$ are in $W_1$. For $h \in \GL_{V_\ell}(\Q_\ell)$, the condition $h(v_i \otimes v_i^\vee) = v_i \otimes v_i^\vee$ for all $i$ amounts to the fact that the matrix of $h$ is diagonal in the basis $\{v_i\}$. We consider $m$ odd and even separately:
    \begin{enumerate}
        \item Suppose that $m$ is odd. By Theorem \cite[Theorem 3.3.3]{part1}, the geometric endomorphism ring of $J_m$ has rank
        \[
        \sum_{\substack{d \mid m,\\ d \neq 1,2}} 2 \dim X_d = \sum_{\substack{d \mid m,\\ d \neq 1}} \varphi(d) = m - 1 = 2g.
        \]
        Equation \Cref{eq: subspace W_1 and endomorphisms} shows that $\dim_{\Q_\ell} W_1 = \operatorname{rk}_{\Z} \End(A_{\overline{K}}) = 2g$, and since each of the $2g$ classes $\psi_k$ is in $W_1$, and they are clearly linearly independent, this shows that $W_1$ is spanned by the $\psi_k$. For $h \in \GL_{V_\ell}(\Q_\ell)$, the condition that $h(W_1) \subset W_1$ works out to $h$ being a generalised permutation matrix (a matrix with precisely one non-zero entry on every row and column) in the basis $\{v_i\}$. Indeed, since $W_1$ is spanned by the classes $v_i \otimes v_i^\vee$, identifying $V_\ell \otimes V_\ell^\vee$ with $\operatorname{End}(V_\ell)$ we see that $W_1$ is the subspace of endomorphisms whose matrix is diagonal in the basis $v_1, \ldots, v_{2g}$. The condition $h(W_1) \subseteq W_1$ means that the matrix of $h$ normalises the torus of diagonal matrices. A simple exercise in linear algebra, or equivalently the fact that the Weyl group of $\GL_{2g}$ is the symmetric group $S_{2g}$, shows that any invertible matrix normalising the torus of diagonal matrices is a generalised permutation matrix.
        
        \item If instead $m$ is even, the geometric endomorphism ring of $J_m$ is larger than the span of the $\psi_k$, and we see relations in degree $1$ (corresponding to classes in $W_1$) showing up as defining equations for the identity component $\Gl^0$ (equivalently, $\operatorname{MT}(J_m)$) inside the diagonal torus of $\GL_{V_\ell}$. For example, for $m=2d$ with $d$ odd, there are relations of the form $x_i/x_j=1$ for suitable $i,\, j$ (see \cite[Example 4.2.11]{part1}).
    \end{enumerate}
    \end{remark}

\begin{remark}[$n=2$ and the polarisation for Fermat Jacobians]\label{rmk: equations coming from the polarisation of Fermat Jacobians}
We will describe below in \Cref{prop: matrix of the polarisation} the matrix (with respect to the basis $\omega_i$) of a bilinear form $\varphi$ on $V$ preserved by $\MT(J_m)$. All we need to know for the moment is that $\varphi(\omega_i, \omega_{m-i}) \neq 0$ for all $i$, which is easy to prove directly. For a diagonal matrix $x=\operatorname{diag}(x_1,\ldots,x_{m-1}) \in \MT(J_m)(\C)$, where the coordinate $x_j$ corresponds to the basis vector $\omega_j$, this implies in particular (assuming $i \neq m/2$ when $m$ is even)
\[
x_ix_{m-i}\varphi(\omega_i, \omega_{m-i}) =\varphi(x\omega_i, x\omega_{m-i}) = \operatorname{mult}(x) \varphi(\omega_i, \omega_{m-i}) \Rightarrow x_i x_{m-i} = \operatorname{mult}(x),
\]
hence $x_i x_{m-i}$ is independent of $i$. This shows that $x$ satisfies the equations $x_i x_{m-i} = x_j x_{m-j}$ for all $i, \, j \in \{1, \ldots, m-1\} \setminus \{m/2\}$. 
Equivalently, the classes $\omega_i \otimes \omega_{m-i} \otimes \omega_j^\vee \otimes \omega_{m-j}^\vee$ are preserved by $\MT(J_m)$, and therefore (since $\omega_i$ is proportional to $v_i$ and $\MT(J_m)(\Q_\ell) = \Gl^0(\Q_\ell)$) all the classes $v_i \otimes v_{m-i} \otimes v_j^\vee \otimes v_{m-j}^\vee$ are in $W_2$. However, as already remarked, $W_2$ could contain more than just these classes (see \Cref{sec: full example ST15}).
\end{remark}

Before moving on to the description of the action of $\Gal(\KconnJ/K)$ on $W_n$ for arbitrary $n$ we point out a general feature of this action. We have seen in \Cref{prop:ShiodaIII}
that the tensor powers of the first cohomology of Fermat Jacobians can be embedded in the cohomology of suitable (higher-dimensional) Fermat varieties.
In turn, the étale cohomology of Fermat varieties decomposes as the direct sum of certain $1$-dimensional eigenspaces (see \Cref{remark: character indexing}), which are permuted by the Galois action in an easily predictable way, as we discuss below in \Cref{lemma: Galois action and permutation of the characters}. This will have interesting consequences for the representation $\rho_W$ in the case of the abelian varieties $J_m$.

To state the lemma, for every $u \in (\Z/m\Z)^\times$ we let $\tau_u$ be the automorphism of $\Q(\zeta_m)$ that sends $\zeta_m$ to $\zeta_m^u$.

\begin{lemma}[Galois action and characters]\label{lemma: Galois action and permutation of the characters}
Let $\omega$ be a class in $\HH^n_{\operatorname{\acute{e}t}}(X^n_m, \Q_\ell)_{\alpha}$. Let $\sigma$ be an automorphism in $\Gal(\overline{\Q}/\Q)$ and let $u \in (\Z/m\Z)^\times$ be such that $\sigma|_{\Q(\zeta_m)}=\tau_u$. The class $\sigma(\omega)$ belongs to the eigenspace $\HH^n_{\operatorname{\acute{e}t}}(X^n_m, \Q_\ell)_{u^{-1}\alpha}$.
\end{lemma}
\begin{proof}
    Since the action of $G_m^n$ on $X_m^n$ is defined over $\Q$, for every $g \in G_m^n(\overline{\Q})$ we have
\[
\sigma(g^* \omega) = \sigma(g)^* \sigma(\omega),
\]
and hence
\[
\sigma(g)^* \sigma(\omega) = \sigma(g^* \omega) = \sigma(\alpha(g) \omega) = \alpha(g) \sigma(\omega).
\]
(For the last equality above, recall that the Galois action on $\ell$-adic étale cohomology is $\Q_\ell$-linear.)
Taking $g=\sigma^{-1}(h)$ for $h \in G_m^n(\overline{\Q})$ we obtain that, for all $h \in G_m^n(\overline{\Q})$, we have
\[
h^* \sigma(\omega) = \alpha( \sigma^{-1}(h) ) \sigma(\omega) = \alpha( \tau_u^{-1}(h) ) \sigma(\omega).
\]
This shows that $\sigma(\omega)$ belongs to the eigenspace corresponding to the character 
\[
h \mapsto \alpha( \tau_u^{-1}(h) ).
\]
A moment's thought reveals that the application of $\alpha$ commutes with $\tau_u$ (since the evaluation of $\alpha$ at a point $(\zeta_0, \ldots, \zeta_{n+1}) \in G_m^n(\overline{\Q})$ is simply a product of some of the coordinates $\zeta_i$, possibly with some exponents), hence the character above is the same as $h \mapsto \tau_u^{-1}(\alpha(h))$, that is, $u^{-1}\alpha$. %
\end{proof}

\section{The Sato-Tate group of Fermat Jacobians} 
In the case of the Fermat Jacobian $J_m$, the tensor space $W$ can be computed explicitly using the equations for the Mumford-Tate group $\MT(J_m)$.
Extending the work of Deligne on Fermat varieties, we give an explicit description of the action of $\Gal(\Q(\varepsilon_{J_m})/\Q)$ on $W$ and, consequently, an algorithm that computes, for every $\tau \in \Gamma_\Q$, an explicit matrix $\rho(\tau)$ giving the action of $\tau$ on $W$ in a suitable basis. In particular, \cite[Theorem 7.15]{Deligne} describes the action of $\Gal(\overline{\Q(\zeta_m)}/\Q(\zeta_m))$ on the cohomology of $X_m^n$, but we need a description of the action of the full Galois group $\Gal(\overline{\Q}/\Q)$. We achieve this for the closely related variety $(X_m^1)^n$ in \Cref{thm: explicit Galois action on Tate classes}.

\subsection{More on the action of \texorpdfstring{$\operatorname{Gal}(\overline{\Q}/\Q)$}{Gal(Qbar/Q)} on Tate classes}{\label{section: compute tau}}

In order to give a complete description of $\mathcal{G}_\ell(\Q_\ell)$ for the case of Fermat Jacobians, it remains to understand the action of $\Gal(\Q(\varepsilon_{J_m})/\Q)$ on $W_{\leq N}$. This section addresses this point.

For simplicity, we treat the case of odd $m$, so that by \Cref{remark: twists} we may identify the relevant pieces of the cohomologies of $J_m$ and $\tilde{J}_m$ in a Galois-equivariant way.
The same arguments can easily be adapted, with minor variations, to the case of even $m$ (once the action of Galois on the cohomology of $\tilde{J}_m$ is known, using the explicit $\overline{\Q}$-isomorphism $\tilde{J}_m \cong J_m$ it is easy to compute the Galois action on the cohomology of $J_m$). 
Furthermore, we fix an auxiliary prime $\ell$ such that $\mu_m \subset \Q_\ell^\times$.%

\Cref{prop:ShiodaIII}(1) gives an embedding of the first cohomology group of $J_m$ (equivalently in this case, $\tilde{J}_m$) into the cohomology of the Fermat curve $X^1_m$.
More precisely, it gives a Galois-equivariant embedding of $V_\ell^\vee$ into the first étale cohomology group $\HH^1_{\text{ét}}(X_m^1, \Q_\ell)$. The same \Cref{prop:ShiodaIII}(1) also shows that the image of the pull-back map corresponds to the subspace $\bigoplus_{i=1}^{m-1} \HH^{1}_{\gamma_i}$ of $\HH^1_{\text{ét}}(X_m^1, \Q_\ell)$, where $\gamma_i$ is the character $(i, i, -2i)$ defined in \Cref{remark: character indexing} and $\HH^1_{\gamma_i} = \HH^1_{\text{ét}}(X_m^1, \Q_\ell)_{\gamma_i}$.
For every even $q$ we deduce an embedding of $(V_\ell^\vee)^{\otimes q}\otimes \,\Q_\ell(\tfrac{q}{2})$ in the $q$-fold tensor product $\bigotimes_{i = 1}^q \HH^1_{\text{ét}}(X_m^1, \Q_\ell)\otimes \,\Q_\ell(\tfrac{q}{2})$.
The image of $(V_\ell^\vee)^{\otimes q}$ corresponds to the subspace
\begin{equation}
    \label{eq: image of Vell in terms of characters}
    A^{\text{ét}} = \bigotimes_{j=1}^q \bigoplus_{i=1}^{m-1} \HH^1_{\gamma_i}\otimes \,\Q_\ell(\tfrac{q}{2})
    = \bigoplus_{\alpha = \gamma_{i_1} \ast \cdots \ast \gamma_{i_q}} \HH^1_{\gamma_{i_1}} \otimes \cdots \otimes \HH^1_{\gamma_{i_q}}\otimes \,\Q_\ell(\tfrac{q}{2}),
\end{equation}
where the sum is taken over all possible concatenations of $q$ characters of the form $\gamma_i$.
Let $X$ denote the product variety $\prod_{i=1}^q X_m^1$.
Via the K\"unneth formula we can embed $\bigotimes_{i = 1}^q \HH^1_{\text{ét}}(X_m^1, \Q_\ell)\otimes \,\Q_\ell(\tfrac{q}{2})$, and therefore also $A^{\text{ét}}$ and $(V_\ell^\vee)^{\otimes q}\otimes\, \Q_\ell(\tfrac{q}{2})$, in $\HH^q_{\text{ét}}( X_{\overline{\Q}}, \Q_\ell)(\tfrac{q}{2})$.
The product group $G = \prod_{i=1}^q G_m^1$ acts on the product variety $X$ component-wise. We get an action of $G_m^n$ on $X$, for $n = 3q-2$, via the homomorphism $G_m^n \to G $ splitting an $(n+2)$-tuple into $q$ blocks of three.
Notice that the embedding of $A^{\text{ét}}$ in $\HH^q_{\text{ét}}( X_{\overline{\Q}}, \Q_\ell)(\tfrac{q}{2})$ is equivariant for both the $G_m^n$-action and the Galois action; by a slight abuse of notation, we will denote by $A^{\text{ét}}$ the image of this embedding.
For every character $\alpha$ as in \Cref{eq: image of Vell in terms of characters}, the restriction $A^{\text{ét}}_{\alpha}  \hookrightarrow \HH^q_{\text{ét}}( X_{\overline{\Q}}, \Q_\ell)(\tfrac{q}{2})_{\alpha}$ of the injection $A^{\text{ét}}  \hookrightarrow \HH^q_{\text{ét}}( X_{\overline{\Q}}, \Q_\ell)(\tfrac{q}{2}) $ to the $\alpha$-eigenspace for the $G_m^n$-action gives an isomorphism. Indeed, the $\alpha$-eigenspace of $A^{\text{ét}}$ is 1-dimensional by definition, and the $\alpha$-eigenspace of $\HH^q_{\text{ét}}( X_{\overline{\Q}}, \Q_\ell)(\tfrac{q}{2})$ %
is $1$-dimensional by \cite[Proposition 5.1.5]{part1}.

The construction of \Cref{eq: image of Vell in terms of characters} %
can be repeated almost verbatim using (algebraic) de Rham cohomology instead of étale cohomology, and we denote by $A^{\text{dR}}$ the resulting subspace of $\HH_{\operatorname{dR}}^q(X_{\overline{\Q}}/\overline{\Q}){(\tfrac{q}{2})}$. %

We are interested in the Galois action on spaces of the form $V_\ell^{\otimes p} \otimes (V_\ell^\vee)^{\otimes p}$ and we know that $V_\ell \cong V_\ell^\vee(1)$, so it suffices to compute the action on $(V_\ell^\vee)^{\otimes 2p}(p)$, which embeds (equivariantly) inside the cohomology group $\HH^{2p}_{\text{ét}}( X_{\overline{\Q}}, \Q_\ell)(p)$, where now $X$ is the product of $q=2p$ copies of $X^1_m$.
Thus, our main objective in this section will be to describe the Galois action on the middle cohomology of the product variety $(X_m^1)^q$, and more precisely on the Tate classes in this cohomology group. %
{For simplicity, we let $\overline{X}$ be the base-change $X \times_{\operatorname{Spec} \Q} \operatorname{Spec} \overline{\Q}$. %
}

\subsection{Preliminaries}
Following \cite{Deligne}, {for every $p \geq 0$} we denote by $\HH^p_{\text{ét}}(\overline{X})$ the `total' étale cohomology of $\overline{X}$, that is,
\[
\HH^p_{\text{ét}}(\overline{X}) \colonequals \HH^p_{\text{ét}}(\overline{X}, \hat{\Z}) = \varprojlim_k \HH^p_{\text{ét}}(\overline{X}, \Z/k\Z).
\]
One can also express $\HH^p_{\text{ét}}(\overline{X})$ as the restricted product of the $\ell$-adic cohomologies $\HH^p_{\text{ét}}(\overline{X}, \Q_\ell)$ with respect to the subgroups $\HH^p_{\text{ét}}(\overline{X}, \Z_\ell)$ {\cite[p.~18]{Deligne}}. We further write
\[
\HH^{2p}_{\mathbb{A}}(\overline{X})(p) \colonequals \HH^{2p}_{\text{ét}}(\overline{X})(p) \times \HH^{2p}_{\operatorname{dR}}(\overline{X} {/ \overline{\Q}})(p).
\]
We will need the notion of \textit{absolute Hodge cycle}, for which we refer the reader to \cite[\S 2]{Deligne}. We denote by
$C^{p}_{\operatorname{AH}}(\overline{X}) \subset \HH^{2p}_{\mathbb{A}}(\overline{X})(p)$ the $\Q$-vector space of codimension-$p$ absolute Hodge cycles on $\overline{X}$.

For any $\Q$-vector space $\HH$ on which the group $G_m^n$ acts, and in particular for any cohomology group of $\overline{X}$ with coefficients in a field of characteristic $0$, there is a canonical decomposition into \textit{generalized eigenspaces}
\begin{equation}\label{eq: decomposition of cohomology groups into generalised eigenspaces}
    \HH = \bigoplus_{[\alpha]} \HH_{[\alpha]},
\end{equation}
where $[\alpha]$ ranges over the $(\Z/m\Z)^\times$-orbits of characters $\alpha$ of $G$ and 
\[
\HH_{[\alpha]} = \{ v \in \HH : v \otimes 1 \in \bigoplus_{u \in (\mathbb{Z}/m\mathbb{Z})^\times} (\HH \otimes_\Q \Q(\zeta_m))_{u\alpha} \}.
\]
The same applies to any $G_m^n$-stable subspace of $\HH$, for example the subspace $A^{\et}$ defined in \Cref{eq: image of Vell in terms of characters} and its de Rham version $A^{\operatorname{dR}}$.
The decomposition \Cref{eq: decomposition of cohomology groups into generalised eigenspaces} applies in particular to the space of absolute Hodge cycles $C^{p}_{\operatorname{AH}}(\overline{X})$. The coordinate ring $\Q[G_m^n]$ of $G_m^n$ also admits a decomposition
\begin{equation}\label{eq: coordinate ring of Gmn}
    \Q[G_m^n] \cong \prod_{[\alpha]} \Q([\alpha]),
\end{equation}
where each $\Q([\alpha])$ is a field whose degree over $\Q$ is the order of the character $\alpha$ {\cite[p.~80]{Deligne}}. For every character $\alpha$ we have a natural map
\[
\begin{array}{ccc}
G_m^n(\overline{\Q}) & \to & \C \\
\underline{\zeta} = \left(\zeta_0, \, \zeta_1, \, \dots, \, \zeta_{n+1}\right) & \mapsto & \underline{\zeta}^\alpha \colonequals \prod_{i=1}^{n+1} \zeta_i^{\alpha_i}
\end{array}
\]
which induces an embedding 
\begin{equation}\label{eq: embedding iota alpha}
\iota_\alpha : \Q([\alpha]) \hookrightarrow \Q(\zeta_m).
\end{equation}
Note that the embedding \Cref{eq: embedding iota alpha} depends not only on the orbit $[\alpha]$, but also on the specific representative $\alpha$.
The space $\HH_{[\alpha]}$ is a $\Q([\alpha])$-vector space.

As already discussed, we are only interested in the case $p=q/2$, that is, in the middle cohomology of $X$. Thus, from now on, we set $p=q/2$. 

\subsection{Reduction to generalised eigenspaces}
{\label{sec: reduction to generalized eigenspaces}}
Recall again that our objective is to compute the Galois action on $A^{\text{ét}} \subseteq\HH^q_{\text{ét}}(\overline{X}, \Q_\ell)(p)$ (see \Cref{eq: image of Vell in terms of characters} for the definition of $A^{\text{ét}}$). %
As discussed, this space decomposes as the direct sum of the generalised eigenspaces $A^{\text{ét}}_{[\alpha]}$ over the orbits of characters of $G_m^n$. It is easy to see (and follows from \Cref{lemma: Galois action and permutation of the characters}) that every generalised eigenspace $A^{\text{ét}}_{[\alpha]}$ is Galois-stable. Thus, we can immediately reduce to studying the subspace $A^{\text{ét}}_{[\alpha]}$ for a fixed orbit $[\alpha]$.

In our application we are only interested in Tate classes, that is, classes in $A^{\text{ét}}_{[\alpha]}$ that are fixed by an open subgroup of the Galois group. More precisely, we are interested in Tate classes derived from equations for the Mumford-Tate group of $J_m$ {(that is, the classes $v_f$ of \Cref{eq: associate tate classes vf})}. %
It follows from \Cref{prop:ShiodaIII} that if one of these classes lies in $A^{\text{ét}}_{[\alpha]}$, then $\alpha=(a_0,\ldots,a_{n+1})$ is an element of $\mathfrak{B}_m^n$, i.e., $\alpha$ satisfies two conditions: no $a_i$ vanishes and $\langle u\alpha \rangle$ is constant for $u \in (\Z/m\Z)^\times$.

\subsection{Strategy for the computation of the Galois action}
{\label{sub: strategy}}
From now on, we work with a fixed character $\alpha=(a_0,\ldots,a_{n+1})$ of $G_m^n$ of the form
$\gamma_{i_1} \ast \dots \ast \gamma_{i_q} $, where $\gamma_i=(i, i, -2i)$ is the character of $G_m^1$ defined in \Cref{remark: character indexing}. We further assume that $\alpha \in \mathfrak{B}_m^n$.

Let $A^{\operatorname{AH}}_{[\alpha]}$ be the $[\alpha]$-generalized eigenspace of $C^p_{\operatorname{AH}}(X_{\overline{\Q}})$. %
The canonical projections $\pi_\ell\colon A^{\operatorname{AH}}_{[\alpha]} \to A^{\text{ét}}_{[\alpha]}$ and $\pi_\infty\colon A^{\operatorname{AH}}_{[\alpha]} \to A^{\operatorname{dR}}_{[\alpha]}$ are injective and equivariant both for the $G_m^n$-action and the Galois action (%
injectivity follows from the definition of absolute Hodge cycles \cite[\S 2]{Deligne}, see also the proof of \cite[Proposition 2.9]{Deligne}).
Note that $A^{\operatorname{AH}}_{[\alpha]}$ is a $\Q([\alpha])$-vector space of dimension $1$. This is \cite[Corollary 7.8]{Deligne} once we replace $X_m^n$ with $X$; the argument works equally well in our context: $A^{\operatorname{dR}}_{[\alpha]}$ is a $\Q([\alpha])$-vector space of dimension $1$, every element is a Hodge class (this follows from \Cref{remark: character indexing}), and every Hodge class is an absolute Hodge class \cite[Corollary 6.27]{Deligne2}. %

The advantage of working with absolute Hodge cycles instead of Tate classes is that we can compute with de Rham components. In fact, describing an explicit basis for the de Rham (rather than étale) cohomology groups is much easier.
We fix a basis for the subspace of $\HH^1_{\operatorname{dR}}(\overline{X_m^1}/\overline{\Q})$ corresponding to the image of the pull-back through the map $\psi\colon X^1_m \to \Tilde{C}_m $ of \cite[Lemma 5.2.3]{part1}.
In the proof of this lemma, we computed the pull-back of (the class of) the differential form $x^i \, dx/y$.
With this in mind, for every $i = 1,\,2, \,  \dots,\, (m-1)$ we let
\begin{equation}
{\label{eq: definition of omegai}}
    \omega_{i} = \psi^\ast(x^{i-1} \, dx/y) =  - x^{i-1} y^{i-m} \; dx.
\end{equation}
We note that these are not literally the same classes we previously denoted by $\omega_i$. However, since $\psi^*$ is injective and equivariant for all the actions we need to consider, we feel that the abuse of notation should not cause any confusion.
Notice that $\omega_i$ is a non-zero element in the one-dimensional eigenspace $\HH^1_{\operatorname{dR}}(\overline{X_m^1}/\overline{\Q})_{\gamma_i}$.
A basis for $A^{\operatorname{dR}}$ is given by the set of all tensors of the form
\begin{equation}
    {\label{eq: definition of omega alpha}}
    \omega_\alpha = \omega_{i_1} \otimes \cdots \otimes \omega_{i_q},
\end{equation}
as each $i_k$ varies in $1,\,2, \,  \dots,\, (m-1)$.
Our strategy to compute the Galois action is as follows. 
\begin{enumerate}
    \item We fix a well-chosen $\Q([\alpha])$-basis $\gamma$ of the $1$-dimensional $\Q([\alpha])$-vector space $A^{\operatorname{AH}}_{[\alpha]} \subseteq C^p_{\operatorname{AH}}(\overline{X})$ (see \Cref{subsubsec: definition of gamma}).
    \item We write $\pi_\infty(\gamma) = \sum c_\beta [\omega_\beta]$ for certain explicit coefficients $c_\beta$, where $\beta$ ranges over the characters in the orbit $[\alpha]$ of $\alpha$ (see \Cref{lemma: coordinates of gamma dR}).
    \item We invert this relation to express each $[\omega_\beta]$ as linear combination of classes of the form $g \cdot \pi_\infty(\gamma)$ for $g$ varying in $G_m^n$ (see \Cref{lemma: writing omegas with gamma and Gmn action}).
    \item For every $\tau \in \Gamma_\Q$, we describe the action of $\tau$ on $A^{\operatorname{AH}}_{[\alpha]}$ purely in terms of the action of $\tau$ on $\gamma$ (see \Cref{lemma: trivialities about the GammaQ action on absolute Hodge cycles} and \Cref{eq: lambda of sigma}).
    \item We deduce from this the Galois action on $A^{\text{ét}}_{[\alpha]} \otimes_{\Q_\ell} \C \subseteq \HH^q_{\text{ét}}(\overline{X}, \Q_\ell)(p)_{[\alpha]} \otimes_{\Q_\ell} \C \cong \HH^q_{\text{dR}}(\overline{X}, \C)(p)_{[\alpha]}$ as represented in the basis $[\omega_\beta] \otimes 1$ (see \Cref{eq: Galois action final form} and \Cref{thm: explicit Galois action on Tate classes}).
\end{enumerate}

\subsection{A special cycle}\label{subsubsec: definition of gamma}
We let $\gamma_0 \in \HH_1^B(X_m^1(\C), \Q)$ be the cycle introduced by Kashio in \cite[Section 3]{Kashio}, with the property that
\begin{equation}
    {\label{eq: Kashio property}}
    \int_{\gamma_0} \omega_i = \Gam{i}{m}^2 \cdot \Gam{2i}{m}^{-1}.
\end{equation}
Notice that $\gamma_0$ is defined for the Fermat curve $F$ with affine equation $x^m+y^m=1$. The isomorphism $F \to X^1_m, \, (x,y) \mapsto (-x,-y)$ changes the sign of the expression $x^{i-1}y^{i-m} \; dx$. The differential form $x^{i-1}y^{i-m} \; dx$ is thus pushed forward to our $\omega_i$ (as defined in \Cref{eq: definition of omegai}) and, with a negligible abuse of notation, we call $\gamma_0$ the push-forward of the original cycle $\gamma_0$. With these minor adjustments, \cite[Equation (3.2)]{Kashio} translates into Equation \Cref{eq: Kashio property} above.
Finally, we define the cycle $\gamma_{\text{top}} = \bigotimes_{i=1}^q\gamma_0$ and consider its class $\gamma_{\text{top}} \otimes (2\pi i)^{-p} $ in $\HH_q^B(X(\C), \Q)(-p)$, where we recall that $2p=q$. {The class $\gamma_{\text{top}}$} has the property that, given $\omega_\alpha = \omega_{i_1} \otimes \cdots \otimes \omega_{i_q} $ as in \Cref{eq: definition of omega alpha}, we have
\begin{equation}
    {\label{eq: Kashio full property}}
    \int_{\gamma_{\text{top}}} \omega_\alpha
    =\prod_{r = 1}^q \Gam{i_r}{m}^2 \Gam{2i_r}{m}^{-1}.
\end{equation}

Starting from the class $\gamma_{\text{top}}\otimes (2\pi i)^{-p}$ in $\HH_q^B(X(\C), \Q)(-\frac{q}{2})$ and a character $\alpha = \gamma_{i_1} \ast \cdots \ast \gamma_{i_q} \in \mathfrak{B}_m^n$, we construct an absolute Hodge class $\gamma_{[\alpha]} \in \HH_{\mathbb{A}}^q(X_{\overline{\Q}})(p)_{[\alpha]} $. 
We follow Deligne's construction in \cite[Section 7]{Deligne}, adding some details to Milne's account.
The first step is to compose the isomorphism induced by the Poincaré pairing with the cycle class map (which is well-defined over the complex numbers): we obtain a morphism 
\[ \HH_q^B(X(\C), \C)(-p) \xrightarrow{\cong} \HH^{q}_B(X_\C, \C)(p) \hookrightarrow \HH_{\mathbb{A}}^{q}(X_\C)(p). \]
The image of $\gamma_{\text{top}}\otimes (2\pi i)^{-p}$ in the absolute cohomology group {$\HH_{\mathbb{A}}^{q}(X_{\C})(p)$} might not be algebraic, but is certainly defined over {some algebraically closed field $k \subset \C$} with finite transcendence degree over $\Q$. In other words, the image of $\gamma_{\text{top}}\otimes (2\pi i)^{-p}$ lies in the subspace $\HH_{\mathbb{A}}^{q}(X_k)(p) \subseteq \HH_{\mathbb{A}}^{q}(X_\C)(p)$.
Denote by $\gamma_{[\alpha]}$ the $[\alpha]$-component of this image.
It is not hard to check that $\gamma_{[\alpha]} \in \HH_{\mathbb{A}}^{q}(X_k)(p) $ is a Hodge cycle (in the sense of \cite[\S 2]{Deligne}): it is rational over $k$ and its de Rham component lies in $\HH_{\operatorname{dR}}^q(X_k/k)_{[\alpha]} \subseteq F^0 \HH_{\operatorname{dR}}^q(X_k/k)$ (see \cite[Proposition 5.1.5]{part1}).
For classes in the cohomology of $X$ it is known that Hodge classes are absolute Hodge \cite[Corollary 6.27]{Deligne2}. In particular, it follows that $\gamma_{[\alpha]}$ is an absolute Hodge class.
Finally, \cite[Proposition 2.9]{Deligne} applies, telling us that the absolute Hodge cycles over $k$ and over $\overline{\Q}$ coincide: $C^{\operatorname{AH}}_p(X_k) = C^{\operatorname{AH}}_p(X_{\overline{\Q}})$. In particular, $\gamma_{[\alpha]}$ is an element of $C^{\operatorname{AH}}_p(X_{\overline{\Q}})$. 

From this point on, we let $\gamma = \gamma_{[\alpha]}$ to ease the notation. Notice that, by definition, $\gamma \in A^{\operatorname{AH}}_{[\alpha]}$.
We know that $\gamma$ is nonzero (by \Cref{eq: Kashio full property} its de Rham component is non-zero, as one sees integrating against the form $\omega_\alpha$), and hence, since $A^{\operatorname{AH}}_{[\alpha]}$ is a 1-dimensional $\Q([\alpha])$-vector space, the element $\gammaAHa$ is a $\Q([\alpha])$-basis of this space. 
To make progress and describe $\pi_\infty(\gammaAHa)$, we look more closely at the map {$\HH_q^B(X(\C), \C)(-p) \to \HH^{q}_{\mathbb{A}}(X_{\C}) \xrightarrow{\pi_\infty} \HH^q_{\operatorname{dR}}(X_{\C}/\mathbb{C})(p)$}: it can be identified with
\[
\HH_q^B(X(\C), \C)(-p) \xrightarrow{\cong} \HH^q_{\operatorname{dR}}(X(\C), \C)(p)^\vee \xrightarrow{\cong} \HH^q_{\operatorname{dR}}(X(\C), \C)(p),
\]
where the first isomorphism sends $c \otimes (2\pi i)^{-p}$ to the linear functional 
\[
\left( \omega \mapsto (2\pi i)^{-p} \int_c \omega \right) \in \HH^q_{\operatorname{dR}}(X(\C), \C)(p)^\vee,
\]
while the second isomorphism is induced by the Poincaré pairing
\begin{equation}
    {\label{eq: Poincare pairing}}
    \begin{array}{ccc}
    \HH^q_{\operatorname{dR}}(X(\C), \C) \times \HH^q_{\operatorname{dR}}(X(\C), \C) & \to &\mathbb{C}(q) \\
    (\omega_1, \omega_2) & \mapsto & \int_{\overline{X}} \omega_1 \wedge \omega_2.
    \end{array}
\end{equation}
\begin{lemma}\label{lemma: Poincare isomorphism}
    The isomorphism $\HH^1_{\operatorname{dR}}(X_m^1(\C), \C)^\vee \xrightarrow{\cong} \HH^1_{\operatorname{dR}}(X_m^1(\C), \C)(1)$ induced by the Poincaré pairing sends $[\omega_i]^\vee$ to $\mu_i [\omega_{-i}]$ for $i =1,\, 2, \,  \dots,\, (m-1)$, where
    \[
        \mu_i = (m-2i)/m
    \]
    More generally, let $\alpha = \gamma_{i_1} \ast \cdots \ast \gamma_{i_q}$ be a character. The isomorphism $\HH^q_{\operatorname{dR}}(\overline{X}(\C), \C)(p)^\vee \xrightarrow{\cong} \HH^q_{\operatorname{dR}}(\overline{X}(\C), \C)(p)$ induced by the Poincaré pairing sends $[\omega_\alpha]^\vee$ to $\mu_{\alpha} [\omega_{-\alpha}]$, where $\mu_{\alpha} = \mu_{i_1} \cdots \mu_{i_q}$.
\end{lemma}

{Technically, the sign of $\mu_i$ depends on which of the two isomorphisms induced by the Poincaré pairing we consider (that is, if we decide to send $\omega \mapsto \omega \wedge \bullet$ or $\omega \mapsto \bullet \wedge \omega$ in \Cref{eq: Poincare pairing}). Since $q$ is assumed to be even, the sign of $\mu_\alpha$ does not depend on this choice.}

\begin{remark}
{\label{rmk: mu alpha is an even function in alpha}}
    By definition of $\mu_i$ (see the statement of \Cref{lemma: Poincare isomorphism}) we have $$\mu_{-i} = \tfrac{m-2(m-i)}{m} = \tfrac{2i-m}{m} = -\tfrac{m-2i}{m} = -\mu_i.$$
    When $\alpha = \gamma_{i_1} \ast \cdots \ast \gamma_{i_q}$, we can simplify the expression for $\mu_{-\alpha}$:
    $$\mu_{-\alpha} = \prod_i \mu_{-i} = \prod_i -\mu_i = (-1)^q \mu_\alpha.$$
    When $q$ is even, which is the case of interest, it follows that $\mu_{-\alpha} = \mu_\alpha$.
\end{remark}

\begin{proof}
    Let $F$ be the curve defined by the affine equation $u^m+v^m=1$, let $X_m^1$ be the Fermat curve given by the affine equation $x^m+y^m+1=0$, and denote by $\psi\colon F \to X_m^1$ the isomorphism given by $(u,v) \mapsto (-x, -y)$.
    In \cite[Section VI, p.~44]{Coleman}, Coleman defines a basis $\{\nu_{r,s}\}$ for $\HH^1_{\operatorname{dR}}(F_{\C}/\C)$, where $(r,\, s) $ varies over the pairs $(r,s) \in \frac{1}{m}\Z/\Z \times \frac{1}{m}\Z/\Z $ such that $r+s \neq 0$.
    In this basis, the cup product is easy to describe: the basis element $\nu_{r,s}$ is orthogonal to every basis element except $\nu_{-r, -s}$, for which
    \begin{equation}
        {\label{eq: cup product on F}}
        \nu_{r,s} \cup \nu_{-r, -s} = (-1)^{\{r\}+ \{s\}-\{r+s\}},
    \end{equation}
    where $\{x\}$ is the representative of an element $x \in \frac{1}{m}\Z/\Z$ in $\Q \cap [0, 1)$.
    Expanding the definition, we observe that
    \begin{equation*}
        {\label{eq: coleman base change}}
        \nu_{ \frac{i}{m}, \frac{i}{m} }
        = K\left(\frac{i}{m}, \frac{i}{m}\right) \cdot u^{i} v^{i} \left(\frac{v}{u}\right) d \left(\frac{u}{v}\right),
    \end{equation*}
    where $K\left(\frac{i}{m}, \frac{i}{m}\right) = (m-2i)/m$ if $i> m/2$ and is 1 otherwise. %
    Using that $mu^{m-1}du+mv^{m-1}dv=0$ on $F$, 
    we can further simplify the expression for the differential form above:
    \begin{align*}
        u^{i} v^{i} \left(\frac{v}{u}\right) \, d \left(\frac{u}{v}\right)
        & =u^{i} v^{i} \left(\frac{v}{u}\right) \left(\frac{1}{v} \, du - \frac{u}{v^2} \, dv \right) \\
        &= u^{i-1} v^{i}\, du + u^{i+m-1} v^{i-m} \, du \\
        & = u^{i-1} v^{i} \, du + u^{i-1} (1-v^m) v^{i-m} \, du \\
        & = u^{i-1} v^{i-m} \, du.
    \end{align*}
    The isomorphism $\psi\colon F \to X_m^1$ pulls back the differential form $\omega_i$ on $X_m^1$ defined in \Cref{remark: character indexing} to $u^{i-1} v^{i-m} du$. Since the pull-back isomorphism commutes with the cup product we get
    \begin{equation}
    {\label{eq: cup products relation}}
        \omega_i \cup \omega_j
        = \left[K\left(\frac{i}{m}, \frac{i}{m}\right)\cdot K\left(\frac{j}{m}, \frac{j}{m}\right)\right]^{-1}  \nu_{ \frac{i}{m}, \frac{i}{m} } \cup \nu_{ \frac{j}{m}, \frac{j}{m} }
    \end{equation}
    and we conclude that the product is 0 unless $j \equiv -i \pmod{m}$. In this case, we assume that $i < m/2$ and explicitly compute, using \Cref{eq: cup product on F} and the definition of $K$, the product 
    \begin{equation}
    {\label{eq: cup product omegai final result}}
        \omega_i \cup \omega_{-i} = \frac{m}{m-2i}.
    \end{equation}
    The statement follows after noticing that the substitution $i \mapsto (m-i)$ exchanges the expressions $(m-2i) $ and $ (2i-m)$.
\end{proof}

In the next lemma, we compute the coordinates of $\pi_\infty(\gammaAHa)$ in the basis $[\omega_\alpha]$. It will be useful to introduce the following notation.
Let $\omega$ be a $C^\infty$ differential $q$-form (where $q=2p$ is an even positive integer) on $X_{\C}$ whose class $[\omega]$ in $\HH_{\operatorname{dR}}^q(X_{\C} / \C)$ lies in $\HH_{\operatorname{dR}}^q(X_{\overline{\Q}} / \overline{\Q})$.
We define
\[
P(\gamma, \omega) = (2\pi i)^{-q/2} \cdot \int_{\gamma_{\text{top}}} \omega.
\]
\begin{lemma}\label{lemma: coordinates of gamma dR}
    With $\mu_{\beta}$ as in \Cref{lemma: Poincare isomorphism}, we have $\pi_\infty(\gammaAHa) = \sum_{\beta \in [\alpha]} \mu_{\beta} P(\gamma, \omega_{-\beta}) [\omega_{\beta}]$. %
\end{lemma}
\begin{proof}
    By \Cref{lemma: Poincare isomorphism}, it suffices 
    to show that the image of $\gammaAHa$ in the $[\alpha]$-generalized eigenspace
    of $\HH^q_{\operatorname{dR}}(X(\C), \C)(p)^\vee$ is $\sum_{\beta \in [\alpha]} P(\gamma, \omega_{-\beta}) [\omega_{-\beta}]^\vee$. Since  $\{[\omega_\beta]\}_{\beta \in [\alpha]}$ is a basis of the dual generalized eigenspace and by definition $\gamma([\omega_\delta])= (2\pi i)^{-q/2}\int_{\gamma_{\operatorname{top}}} \omega_\delta$, the statement boils down to the equality
    \[
    \left(\sum_{\beta \in [\alpha]} P(\gamma, \omega_{-\beta}) [\omega_{-\beta}]^\vee\right)(\omega_\delta) = (2\pi i)^{-q/2}\cdot\int_{\gamma_{\text{top}}} \omega_\delta
    \]
    for all $\delta \in [\alpha]$, and this is true by definition. Finally, notice that $\mu_{-\beta} = \mu_\beta$ by \Cref{rmk: mu alpha is an even function in alpha}.
\end{proof}

{Note that the de Rham class $[\omega_\beta]$ is even defined over $\Q$. Since we want to stress the difference between cohomology over $\Q$ and over $\overline{\Q}$, from now on we write $[\omega_\beta] \in \HH^q_{\operatorname{dR}}(X/\Q)$ and $[\omega_\beta] \otimes 1 \in \HH^q_{\operatorname{dR}}(X_{\overline{\Q}}/\overline{\Q})$.}
Recall that $[\omega_\beta] \otimes 1$ lies in the $\beta$-eigenspace for the action of $G_m^n$: hence, for every $g \in G_m^n(\overline{\Q})$ we have $g \cdot ([\omega_\beta] \otimes 1) = [\omega_\beta] \otimes  \beta(g)$.

We note that $P(\gamma, \omega_\beta)$ is non-zero for every $\beta \in [\alpha]$ (see the explicit formula \Cref{eq: Kashio full property}). For a character $\alpha = \gamma_{i_1} \ast \cdots \ast \gamma_{i_q}$ in $\mathfrak{B}_m^n$, the number $P(\gamma, \omega_{\alpha})$ is algebraic: we give two proofs. Expanding the definition and using \Cref{eq: Kashio full property}, the integral evaluates to
\begin{equation}\label{eq: evaluation of P gamma omega-alpha}
    P(\gamma, \omega_\alpha) = (2\pi i)^{-p} \cdot \prod_{r = 1}^q \Gam{i_r}{m}^2 \Gam{2i_r}{m}^{-1}.
\end{equation} 

A standard computation with the reflection formula shows that, up to an algebraic constant, the right-hand side equals $\Gamma(\alpha)$, which is algebraic by \cite[Theorem 6.1.3]{part1}.
Alternatively, we know that $\pi_\infty(\gamma) \in \HH^{q}_{\operatorname{dR}}(\overline{X}/\overline{\Q})$, and the $P(\gamma, \omega_\beta)$ are its coordinates along the $\Q$-rational basis $[\omega_\alpha]$, which again proves $P(\gamma, \omega_\beta) \in \overline{\Q}$.

We now express $[\omega_\beta]$ purely in terms of $\pi_\infty(\gammaAHa)$ and the $G_m^n$-action. 

\begin{lemma}\label{lemma: writing omegas with gamma and Gmn action}
    With $\mu_{\beta}$ as in \Cref{lemma: Poincare isomorphism}, for every $\beta \in [\alpha]$ we have the equality
    \begin{equation}\label{eq: omega in terms of gamma}
    [\omega_\beta] \otimes 1 = \frac{1}{ \#G_m^n(\overline{\Q})} \sum_{g \in G_m^n(\overline{\Q})} (g \cdot\gamma) \otimes \frac{\beta(g)^{-1}}{ \mu_{\beta} \cdot P(\sigma, \omega_{-\beta})}
\end{equation}
    in $C^p_{\operatorname{AH}}(\overline{X}) \otimes \overline{\Q} = \HH_{\operatorname{dR}}^{2p}(X/\Q) \otimes \overline{\Q}$.
\end{lemma}
\begin{proof}
Straightforward consequence of \Cref{lemma: coordinates of gamma dR} and the orthogonality relations for characters.
\end{proof}

Similarly, \Cref{lemma: coordinates of gamma dR} 
can be interpreted as saying that in  $C^p_{\operatorname{AH}}(\overline{X}) \otimes \overline{\Q} = \HH_{\operatorname{dR}}^{2p}(X/\Q) \otimes \overline{\Q}$ we have the equality
\begin{equation}\label{eq: gamma in terms of omega}
    \gamma \otimes 1 =  \sum_{\beta \in [\alpha]}  [\omega_{\beta}] \otimes (\mu_{\beta} \cdot P(\gamma, \omega_{-\beta})).
\end{equation}

\subsection{The Galois action on \texorpdfstring{$A^{\operatorname{AH}}_{[\alpha]}$}{absolutely Hodge classes}} 
Let $L$ be the number field generated by the finitely many algebraic numbers $\zeta_m$ and $P(\gamma, \omega_\beta)$ for $\beta \in [\alpha]$.
Choose a prime $\ell$ that is totally split in $L$, so that we can make sense of the equalities \Cref{eq: omega in terms of gamma} and \Cref{eq: gamma in terms of omega} also in $A^{\operatorname{AH}}_{[\alpha]}\otimes \Q_\ell$. {In particular, such a prime $\ell$ is congruent to $1$ modulo $m$, so that all the previous considerations apply.}

We are now almost ready to compute the Galois action on $[\omega_\beta] \otimes 1$, deducing it from the Galois action on absolute Hodge cycles. The latter is described in the following proposition.

\begin{lemma}
\label{lemma: trivialities about the GammaQ action on absolute Hodge cycles}
    The following hold:
    \begin{enumerate}
        \item there is a unique map
        \[
        \lambda : \Gamma_\Q \to \Q([\alpha])^\times
        \]
        such that $\tau(\gammaAHa) = \lambda(\tau) \cdot \gammaAHa$ for all $\tau \in \Gamma_\Q$, where the product is given by the vector space structure.
        \item the Galois action on $A^{\operatorname{AH}}_{[\alpha]}$ is determined by the formula
        \[
        \tau(g \cdot \gammaAHa) =
        \tau(g) \cdot \tau(\gammaAHa) = g^u \cdot \lambda(\tau) \cdot \gammaAHa,
        \]
        where $u \in (\Z/m\Z)^\times$ is defined by the condition $\tau(\zeta_m)=\zeta_m^u$.
    \end{enumerate}
\end{lemma}
\begin{proof}
    \begin{enumerate}
        \item Both $\gammaAHa$ and $\tau(\gammaAHa)$ are non-zero elements of a $1$-dimensional $\Q([\alpha])$-vector space, so they are related to one another by multiplication by a non-zero element of the field $\Q([\alpha])$.
        \item Follows from the fact that the action of $G_m^n$ on $X$ is defined over $\Q$. %
    \end{enumerate}
\end{proof}

We are now finally ready to express the Galois action on $[\omega_\beta] \otimes 1$: let $\tau \in \Gamma_\Q$ induce $\zeta_m \mapsto \zeta_m^u$ on $\Q(\zeta_m)$. Using \Cref{eq: omega in terms of gamma} and \Cref{eq: gamma in terms of omega}, as well as the properties of $\lambda(\tau)$ from \Cref{lemma: trivialities about the GammaQ action on absolute Hodge cycles}, we have
\[
\begin{aligned}
    \tau([\omega_\beta] \otimes 1) & = \frac{1}{ \#G_m^n(\overline{\Q})} \sum_{g \in G_m^n(\overline{\Q})} \tau(g \cdot\gammaAHa) \otimes \frac{\beta(g)^{-1}}{ \mu_{\beta} \cdot P(\gamma, \omega_{-\beta})} \\
    & = \frac{1}{ \#G_m^n(\overline{\Q})} \sum_{g \in G_m^n(\overline{\Q})} \tau(g) \cdot \tau(\gammaAHa) \otimes \frac{\beta(g)^{-1}}{ \mu_{\beta} \cdot P(\gamma, \omega_{-\beta})} \\
    & = \frac{1}{ \#G_m^n(\overline{\Q})} \sum_{g \in G_m^n(\overline{\Q})}  g^u \cdot \lambda(\tau) \cdot \gammaAHa \otimes \frac{\beta(g)^{-1}}{ \mu_{\beta} \cdot P(\gamma, \omega_{-\beta})} \\
    & = \frac{1}{ \#G_m^n(\overline{\Q})} \sum_{g \in G_m^n(\overline{\Q})}  g^u \cdot \lambda(\tau) \cdot \left( \sum_{\delta \in [\alpha]}  [\omega_{\delta}] \otimes \left(\mu_{\delta} \cdot P(\gamma, \omega_{-\delta}) \cdot \frac{\beta(g)^{-1}}{ \mu_{\beta} \cdot P(\gamma, \omega_{-\beta})} \right) \right) \\
    & = \frac{1}{ \#G_m^n(\overline{\Q})} \sum_{\delta \in [\alpha]} \lambda(\tau) \cdot  \sum_{g \in G_m^n(\overline{\Q})}  g^u \cdot  [\omega_{\delta}] \otimes   \left( \beta(g)^{-1} \cdot \frac{\mu_{\delta} P(\gamma, \omega_{-\delta})}{ \mu_{\beta} P(\gamma, \omega_{-\beta})} \right).
\end{aligned}
\]
We now recall that $[\omega_\delta] \otimes 1$ lies in the $\delta$-eigenspace of the action of $g$, hence $g^u \cdot [\omega_\delta] \otimes 1 = [\omega_\delta] \otimes \delta(g)^u$. From this we then obtain
\[
\begin{aligned}
    \tau([\omega_\beta] \otimes 1) & = \frac{1}{ \#G_m^n(\overline{\Q})} \sum_{\delta \in [\alpha]} \lambda(\tau) \cdot  \sum_{g \in G_m^n(\overline{\Q})}  g^u \cdot  [\omega_{\delta}] \otimes   \left( \beta(g)^{-1} \cdot \frac{\mu_{\delta} P(\gamma, \omega_{-\delta})}{ \mu_{\beta} P(\gamma, \omega_{-\beta})} \right) \\
    & =\frac{1}{ \#G_m^n(\overline{\Q})} \sum_{\delta \in [\alpha]} \lambda(\tau) \cdot \left(  [\omega_{\delta}] \otimes \frac{\mu_{\delta} P(\gamma, \omega_{-\delta})}{ \mu_{\beta} P(\gamma, \omega_{-\beta})} \sum_{g \in G_m^n(\overline{\Q})}   (u\delta-\beta)(g) \right). 
\end{aligned}
\]
Note that here $u\delta-\beta$ is the character $g \mapsto \delta(g)^u/\beta(g)$, but we use the additive notation.
By orthogonality of characters, the inner sum evaluates to $0$ unless $u\delta=\beta$ %
, and to $\#G_m^n(\overline{\Q})$ otherwise. Thus, only one term in the sum over $\delta$ is non-zero, namely, that corresponding to the character ${u^{-1}}\beta$, where $u^{-1}$ should be interpreted as an element of $(\Z/m\Z)^\times$. We thus obtain the expression
\begin{equation}
\label{eq: almost explicit Galois action}
\begin{aligned}
    \tau([\omega_\beta] \otimes 1) &= \lambda(\tau) \cdot [\omega_{u^{-1}\beta}] \otimes \frac{ \mu_{u^{-1}\beta} P(\gamma, \omega_{-u^{-1}\beta})}{\mu_{\beta} P(\gamma, \omega_{-\beta})}\\
    &= [\omega_{u^{-1}\beta}] \otimes \iota_{u^{-1}\beta}(\lambda(\tau)) \cdot \frac{ \mu_{u^{-1}\beta} P(\gamma, \omega_{-u^{-1}\beta})}{\mu_{\beta} P(\gamma, \omega_{-\beta})},
\end{aligned}
\end{equation}
where $\iota_{u^{-1}\beta}$ is defined in \Cref{eq: embedding iota alpha}.
Note that this is compatible with \Cref{lemma: Galois action and permutation of the characters}: the image of the $\beta$-eigenspace is contained in the eigenspace corresponding to the character $u^{-1}\beta$.

We also note that this is an equality in $C^p_{\operatorname{AH}}(\overline{X})_{[\alpha]} \otimes \Q_\ell = \HH_{\text{ét}}^{2p}(\overline{X}, \Q_\ell)(p)_{[\alpha]}$, hence it implies an identical equality in
\[
\HH^{2p}_{\text{dR}}(X_{\C} / \C)(p)_{[\alpha]} = C^p_{\operatorname{AH}}(\overline{X})_{[\alpha]} \otimes_{\Q} \C = C^p_{\operatorname{AH}}(\overline{X})_{[\alpha]} \otimes_{\Q} \Q_\ell \otimes_{\Q_\ell} \C = \HH^{2p}_{\text{ét}}(\overline{X}, \Q_\ell)(p)_{[\alpha]} \otimes_{\Q_\ell} \C,
\]
which is what we needed: a description of the Galois action in terms of the basis $\{[\omega_\beta] \otimes 1\}$ of $\HH^{2p}_{\text{dR}}(\overline{X} / \C)_{[\alpha]}$. {Also note that we have proved this for a special class of primes $\ell$, but since we have $C^p_{\operatorname{AH}}(\overline{X})_{[\alpha]} \otimes_\Q \Q_\ell = \HH^{2p}_{\operatorname{\acute{e}t}}(\overline{X}, \Q_\ell)(p)_{[\alpha]}$ for all primes $\ell$, the same result holds for all primes.}

To make \eqref{eq: almost explicit Galois action} completely explicit, it only remains to describe the map $\lambda$ and express $P(\gamma, \omega_{-\beta}), \, P(\gamma, \omega_{-u^{-1}\beta})$ in computable terms. 
{The map $\lambda$
can be described as follows: by definition of $\lambda(\tau)$, we have $\tau(\gammaAHa \otimes 1) = (\lambda(\tau) \cdot \gammaAHa) \otimes 1$. We take the projection $\pi_\infty$ on the de Rham component and then apply the canonical map $\HH^n_{\operatorname{dR}}(\overline{X}/\overline{\Q}) \to \HH^n_{\operatorname{dR}}(\overline{X}/\overline{\Q})_\alpha$. Both maps are Galois- and $G_m^n$-equivariant. By \Cref{lemma: coordinates of gamma dR}, and recalling that $[\omega_\alpha]$ lies in $\HH^n_{\operatorname{dR}}(\overline{X}/\Q)$, we obtain%
\[
(\gammaAHa \otimes 1)_{\alpha} = \mu_{\alpha}P(\gamma, \omega_{-\alpha})[\omega_\alpha]
\]
and
\[
\begin{aligned}
    (\lambda(\tau) \gammaAHa \otimes 1)_{\alpha}
&= (\tau(\gammaAHa)\otimes 1 )_{\alpha}\\
&= \tau\left( (\gamma \otimes 1)_{\alpha} \right)\\
&= \tau(\mu_{\alpha} P(\gamma, \omega_{-\alpha})) \tau([\omega_\alpha])\\
&= \tau(\mu_{\alpha} P(\gamma, \omega_{-\alpha})) [\omega_\alpha].
\end{aligned}
\]
On the other hand, by definition of the embedding $\iota_{\alpha}$ we also have
\[
(\lambda(\tau) \gammaAHa \otimes 1)_{\alpha} = \iota_\alpha(\lambda(\tau)) (\gammaAHa \otimes 1)_{\alpha} = \iota_\alpha(\lambda(\tau)) \mu_{\alpha} P(\gamma, \omega_{-\alpha})[\omega_\alpha].
\]
Comparing the above expressions and using that $\mu_{\alpha}$ is rational we then obtain
\begin{equation}\label{eq: lambda of sigma}
    \iota_\alpha(\lambda(\tau)) = \frac{\tau P(\gamma, \omega_{-\alpha})}{P(\gamma, \omega_{-\alpha})}.
\end{equation}
Combining this formula with \Cref{eq: almost explicit Galois action}, we obtain our final expression
\begin{align}
    \tau([\omega_\beta]\otimes 1) &= [\omega_{u^{-1}\beta}] \otimes \frac{\mu_{u^{-1}\beta} }{\mu_{\beta}} \cdot \frac{ P(\gamma, \omega_{-u^{-1}\beta})}{ P(\gamma,\omega_{-\beta})} \cdot \frac{\tau(P(\gamma, \omega_{-u^{-1}\beta}))}{P(\gamma, \omega_{-u^{-1}\beta})} \nonumber \\
    &= [\omega_{u^{-1}\beta}] \otimes \frac{\mu_{u^{-1}\beta} }{\mu_{\beta}} \cdot \frac{ \tau(P(\gamma, \omega_{-u^{-1}\beta}))}{ P(\gamma,\omega_{-\beta})}. \label{eq: Galois action final form}
\end{align}
Finally, we relate the quantities $P(\gamma, \omega_{-\beta})$ and $P(\gamma, \omega_{-u^{-1}\beta})$ to values of the Gamma function, via \Cref{eq: Kashio full property}. This makes \Cref{eq: Galois action final form}, and hence the Galois action, completely explicit.

{
We summarise these results in the next theorem.
\TheoremGaloisActionTateClasses
}

\begin{remark}\label{rmk: formulas without factor mu} As noted in the introduction, one could rescale the basis $[\omega_\beta] \otimes 1$ by a factor $\mu_{\beta}$ to obtain a cleaner formula in \Cref{thm: explicit Galois action on Tate classes}.
A different way to get a simpler formula would be to consider `normalized' gamma products of the form
\[
\hat{P}(\gamma, \omega_\beta) = (2\pi i)^{-n/2} \prod_{r=1}^n \Gamma\left( \frac{i_r}{m} \right)^2 \Gamma\left( \left\{ \frac{2i_r}{m} \right\} \right)^{-1},
\]
where $\{\,\cdot\,\}$ denotes the fractional part of a rational number (taken to be in $(0,1]$). With this notation, one has simply
\[
        \tau([\omega_\beta]\otimes 1) = [\omega_{u(\tau)^{-1}\beta}] \otimes \frac{ \tau(\hat{P}(\gamma, \omega_{-u(\tau)^{-1}\beta}))}{ \hat{P}(\gamma,\omega_{-\beta})} \quad \forall \tau \in \Gal(\overline{\Q}/\Q).
\]
For a proof of this, see below in \Cref{subsect: generalized Gross Koblitz}.
The value $\hat{P}(\gamma, \omega_\beta)$ should be considered as a more natural representative of a period of $\omega_\beta$ in $(X_m^1)^n$. It is usual to compute periods only up to rational constants, but the explicit nature of our computation required us to choose representatives. We decided to fix a natural basis for the homology and the cohomology groups of $(X_m^1)^n$ and consequently define the period $P(\gamma, \omega_\beta)$ as the image of their Poincaré pairing.
\end{remark}

\addtocontents{toc}{\protect\setcounter{tocdepth}{1}}
\section{A complete example: \texorpdfstring{$m=15$}{m=15}}{\label{sec: full example ST15}}
We apply the theory developed in the last chapter to compute the Sato-Tate group of $A=J_{15}$, the Jacobian of the smooth projective curve over $\Q$ with affine equation $y^2=x^{15}+1$.
Our strategy to compute the Sato-Tate group relies on \Cref{thm: description of Gl in terms of Tate classes} and \Cref{thm: explicit Galois action on Tate classes}. We will compute $W_{\leq N}$ and the tensor representation $\rho_{W_{\leq N}}\colon \Gal(\Q(\varepsilon_A)/\Q)\to {W_{\leq N}}$. We will get explicit equations for the connected components of the monodromy group $\Gl$ as algebraic subvarieties of $\GL(V_\ell)$. The coefficients of these equations lie in the connected monodromy field $\Q(\varepsilon_A)$. The Sato-Tate group $\ST(A)$ can then be realized as a (maximal compact) subgroup of $\Gl^1 \times_{\iota} \C$, once we fix an embedding $\iota \colon \Q(\varepsilon_A) \hookrightarrow \C$ (see \Cref{def: Sato-Tate group}).
The reader can find code to verify our computations at \cite{OurScripts}.

    \subsection*{Equations for \texorpdfstring{$\MT(A)$}{MT(A)}} We computed equations defining the Mumford-Tate group $\MT(A)$ in \cite[Example 4.2.9]{part1}. The sum of the positive exponents is $2$ for every equation: by \Cref{lemma: degree of equations is the same as degree of classes}
    we can take $N=2$ in \Cref{cor: replace W by its truncation}. The space $W_1$ is described in \Cref{rmk: n=1 and endomorphisms for Fermat Jacobians}, while $W_2$ is generated by the {class of the polarization} and the classes corresponding to the following equations defining $\MT(A)$: 
    $$x_{9}x_{12}/x_{8}x_{13}=1, \quad x_{11}x_{12}/x_{9}x_{14}=1, \quad x_{10}x_{12}/x_{8}x_{14}=1.$$
    
    \subsection*{From equations to characters} In \Cref{section: compute tau}, we described how to embed the tensor representation $V_\ell^{\otimes p} \otimes (V_\ell^\vee)^{\otimes p}$ in the cohomology $\HH_{\text{ét}}^{2p}(\overline{X}, \Q_\ell)(p)$ of the $2p$-fold product $X = \prod_{i=1}^{2p} X_m^1$, so that the embedding is Galois equivariant. We denoted the image $A^{\text{ét}}$.
    The image of the subspace $W_p\subseteq V_\ell^{\otimes p} \otimes (V_\ell^\vee)^{\otimes p}$ in $A^{\text{ét}}$ can be identified to the sum of the eigenspaces $A^{\text{ét}}_{\alpha}$ over all characters $\alpha$ (of the group $G_m^n$ defined in \Cref{remark: character indexing}) satisfying two conditions (here $n = 6p-3$):
    \begin{enumerate}[label=(\roman*)]
        \item $\alpha$ is the concatenation $\gamma_{i_1} \ast \dots \ast \gamma_{i_{2p}}$ of $2p$ characters of the form $\gamma_i=(i, i, -2i)$, for some $i$ not divisible by $m$ (see \Cref{remark: character indexing});
        \item the weight $\langle u\alpha \rangle$ is constant, as $u$ varies in $(\Z/m\Z)^\times$ (this is equivalent to asking that $\alpha \in \mathfrak{B}_m^n$, after the first condition is satisfied).
    \end{enumerate}
    In our case the indices $i_1, \, \dots, \, i_{2p}$ in a valid character $\alpha$ are the indices of the variables in an equation for the Mumford-Tate group: for example to the equation $x_{9}x_{12}/x_{8}x_{13}=1$ corresponds the set of indices $(9, 12, 15-8, 15-13)=(9, 12, 7, 2)$.

    \subsection*{Orbits of characters} The Galois action on $W_p$ decomposes it into irreducible representations over $\Q$, which we understand in terms of generalized eigenspaces $A^{\text{ét}}_{[\alpha]} = \bigoplus_{u \in \Z/m\Z^\times} A^{\text{ét}}_{u\alpha} $, as described in \Cref{sec: reduction to generalized eigenspaces}.
    Concretely, this means that, given an element $\tau \in \Gal(\Q(\varepsilon_A)/\Q)$, the matrix $\rho(\tau)$ will decompose into square blocks (in a suitable basis for $W_p$), each corresponding to a generalized eigenspace $A^{\text{ét}}_{[\alpha]}$.
    
    \subsection*{Computation of \texorpdfstring{$\rho(\tau)$}{rho(tau)}}
    In \Cref{section: compute tau} we give an explicit basis $\{[\omega_\beta] \otimes 1\}_{\beta \in [\alpha]}$ for $A^{\text{ét}}_{[\alpha]} \otimes_\iota \C$. We compute the matrix $\rho(\tau)$ in this basis for the few characters $\alpha$ described in step (2) using \Cref{thm: explicit Galois action on Tate classes}.

    \subsection*{Computation of \texorpdfstring{$\Gl$}{Gl}}
    We recover equations for the monodromy group $\Gl$ by appealing to \Cref{thm: description of Gl in terms of Tate classes} and computing every component separately. Fix an element $\tau \in \Gal(\Q(\varepsilon_A)/{\Q})$. Any matrix $h$ in the $\tau$-component of the Sato-Tate group is a generalized permutation matrix (\Cref{rmk: n=1 and endomorphisms for Fermat Jacobians}) whose shape is determined uniquely by the restriction of $\tau$ to $\Gal(\Q(\zeta_m)/\Q)$ (\Cref{lemma: Galois action and permutation of the characters}). By imposing that (the tensor powers of) $h$ agree with $\rho_{W_{\leq 2}}(\tau)$ on $ W_{\leq 2} $, we get algebraic equations between the non-zero entries of $h$. These equations cut out the $\tau$-component of the monodromy group group in $\GL_{14}(\Q(\varepsilon_A))$.

    \subsection*{The result}  
    We fix an embedding of $\Q(\varepsilon_A)$ in $\C$ and describe the Sato-Tate group $\ST(A)$ by giving explicit equations for the connected component of the identity and exhibiting matrices corresponding to two generators of $\Gal(\Q(\varepsilon_A)/{\Q})$. 
    \begin{itemize}
        \item The component corresponding to $\tau = \text{id} \in \Gal(\Q(\varepsilon_A)/{\Q})$ is the subgroup of diagonal $14 \times 14$ diagonal complex matrices with entries of absolute value 1, with non-zero entries $x_1, \dots, x_{14}$ cut out by the equations
        \begin{align*}
    x_{1}\cdot x_{14} & = 1,\\
    x_{2}\cdot x_{13} &= 1,\\
    x_{3}\cdot x_{12} &= 1,\\
    x_{4}\cdot x_{11} &= 1,\\
    x_{5}\cdot x_{10} &= 1,\\
    x_{6}\cdot x_{9} &= 1,\\
    x_{7}\cdot x_{8} &= 1,\\
    x_{5} &= x_{3}\cdot x_{4}\cdot x_{13} ,\\
    x_{6} &= x_{3}\cdot x_{4}\cdot x_{14},\\
    x_{7} &= x_{3}^2\cdot x_{4}\cdot x_{13}\cdot x_{14};
        \end{align*}    
        this agrees with \cite[Proposition 5.5]{Heidi}.
        
        \item We select two generators $\tau_1, \, \tau_2$ for $\Gal(\Q(\varepsilon_A)/{\Q})$ such that
        \[ \Gal(\Q(\varepsilon_A)/\Q) = \langle \tau_1, \, \tau_2 \mid \tau_1^8=1, \, \tau_2^4=1, \,\tau_2\tau_1^2\tau_2 =1, \, (\tau_1\tau_2^{-1})^2 = 1 \rangle,  \]
        and $\Q(\zeta_m)/\Q$ corresponds to the subgroup $\langle \tau_1^4 \rangle$ via Galois correspondence. 
        We then compute equations for the respective components and select a matrix of order 8 for each. %
        We get two matrices $M_1, \, M_2$, with the following shape 
        \setcounter{MaxMatrixCols}{14}
        \begin{equation*}
        M_1 = \left[
\begin{smallmatrix}
0 & 0 & 0 & 0 & 0 & 0 & 0 & \ast & 0 & 0 & 0 & 0 & 0 & 0 \\
\ast & 0 & 0 & 0 & 0 & 0 & 0 & 0 & 0 & 0 & 0 & 0 & 0 & 0 \\
0 & 0 & 0 & 0 & 0 & 0 & 0 & 0 & \ast & 0 & 0 & 0 & 0 & 0 \\
0 & \ast & 0 & 0 & 0 & 0 & 0 & 0 & 0 & 0 & 0 & 0 & 0 & 0 \\
0 & 0 & 0 & 0 & 0 & 0 & 0 & 0 & 0 & \ast & 0 & 0 & 0 & 0 \\
0 & 0 & \ast & 0 & 0 & 0 & 0 & 0 & 0 & 0 & 0 & 0 & 0 & 0 \\
0 & 0 & 0 & 0 & 0 & 0 & 0 & 0 & 0 & 0 & \ast & 0 & 0 & 0 \\
0 & 0 & 0 & \ast & 0 & 0 & 0 & 0 & 0 & 0 & 0 & 0 & 0 & 0 \\
0 & 0 & 0 & 0 & 0 & 0 & 0 & 0 & 0 & 0 & 0 & \ast & 0 & 0 \\
0 & 0 & 0 & 0 & \ast & 0 & 0 & 0 & 0 & 0 & 0 & 0 & 0 & 0 \\
0 & 0 & 0 & 0 & 0 & 0 & 0 & 0 & 0 & 0 & 0 & 0 & \ast & 0 \\
0 & 0 & 0 & 0 & 0 & \ast & 0 & 0 & 0 & 0 & 0 & 0 & 0 & 0 \\
0 & 0 & 0 & 0 & 0 & 0 & 0 & 0 & 0 & 0 & 0 & 0 & 0 & \ast \\
0 & 0 & 0 & 0 & 0 & 0 & \ast & 0 & 0 & 0 & 0 & 0 & 0 & 0 
\end{smallmatrix}  \right],
\qquad M_2 = \left[
\begin{smallmatrix}
0 & 0 & 0 & 0 & 0 & 0 & \ast & 0 & 0 & 0 & 0 & 0 & 0 & 0 \\
0 & 0 & 0 & 0 & 0 & 0 & 0 & 0 & 0 & 0 & 0 & 0 & 0 & \ast \\
0 & 0 & 0 & 0 & 0 & \ast & 0 & 0 & 0 & 0 & 0 & 0 & 0 & 0 \\
0 & 0 & 0 & 0 & 0 & 0 & 0 & 0 & 0 & 0 & 0 & 0 & \ast & 0 \\
0 & 0 & 0 & 0 & \ast & 0 & 0 & 0 & 0 & 0 & 0 & 0 & 0 & 0 \\
0 & 0 & 0 & 0 & 0 & 0 & 0 & 0 & 0 & 0 & 0 &  \ast & 0 & 0 \\
0 & 0 & 0 & \ast & 0 & 0 & 0 & 0 & 0 & 0 & 0 & 0 & 0 & 0 \\
0 & 0 & 0 & 0 & 0 & 0 & 0 & 0 & 0 & 0 & \ast & 0 & 0 & 0 \\
0 & 0 & \ast  & 0 & 0 & 0 & 0 & 0 & 0 & 0 & 0 & 0 & 0 & 0 \\
0 & 0 & 0 & 0 & 0 & 0 & 0 & 0 & 0 & \ast & 0 & 0 & 0 & 0 \\
0 & \ast & 0 & 0 & 0 & 0 & 0 & 0 & 0 & 0 & 0 & 0 & 0 & 0 \\
0 & 0 & 0 & 0 & 0 & 0 & 0 & 0 & \ast & 0 & 0 & 0 & 0 & 0 \\
\ast & 0 & 0 & 0 & 0 & 0 & 0 & 0 & 0 & 0 & 0 & 0 & 0 & 0 \\
0 & 0 & 0 & 0 & 0 & 0 & 0 & \ast & 0 & 0 & 0 & 0 & 0 & 0 
\end{smallmatrix}  \right].
\end{equation*}
        The non-zero coefficients of $M_1$ are, from left to right,  
        \[
        \tfrac{13\sqrt{3}}{9} ,
        -3\sqrt{3} , 13\sqrt{3} ,
        1,
        -\tfrac{13i}{3 \sqrt{15}},
        1, 
        1, 
        -13, 
        -3, 
        -\tfrac{3i\sqrt{15}}{13},
        -\tfrac{1}{7}, 
        \tfrac{\sqrt{3}}{13}, 
        -\tfrac{11\sqrt{3}}{63}, 
        \tfrac{3\sqrt{3}}{11}.
        \]
        The non-zero coefficients of $M_2$ are, from left to right,
        \[
        1,
        \tfrac{11\sqrt{3}}{21},
        -\tfrac{11\sqrt{3}}{39},
        -\tfrac{13\sqrt{3}}{3},
        i,
        3,
        13,
        1,
        1,
        -i,
        -\tfrac{7\sqrt{3}}{13},
        \tfrac{39\sqrt{3}}{11},
        -\sqrt{3},
        -\tfrac{11}{13}.
        \]
\end{itemize}

\addtocontents{toc}{\protect\setcounter{tocdepth}{2}}
\section{Complements}
{\label{section: complements}}
We derive from our work in the previous sections, and in \cite{part1}, three more results of independent interest: a description of the Sato-Tate group of certain twists of $J_m$, the computation of the matrix of the canonical polarization of $J_m$, and a generalization of a formula of Gross-Koblitz. Although we don't use these results directly in our study of $J_m$, they allow us to independently verify some of our computations.

\subsection{Twists of \texorpdfstring{$J_m$}{Jm}}\label{subsect: twists}

Fix an integer $m \geq 3$. Let $a \in \Q^\times$ and write $C_{m, a}$ for the smooth projective curve with affine equation $y^2 = x^m + a$; it is clearly a twist of our curve $C_m$. \Cref{thm: description of Gl in terms of Tate classes} shows that, in order to describe the $\ell$-adic monodromy group of $\operatorname{Jac}(C_{m, a})$, it is enough to understand the space of Tate classes $W_{\leq N}$ and the Galois action on this space. By definition, $W_{\leq N}$ depends only on $\operatorname{Jac}(C_{m, a})_{\overline{\Q}} \cong J_{m, \overline{\Q}}$, and so we already have a description of this space and we know that it is independent of $a$. The Galois action is described in \cite[Section 5.3]{part1}. This shows that our algorithm can also be applied to the Jacobian of any twist $C_{m,a}$, and so these results largely complete the project of classifying the Sato-Tate groups of the trinomial hyperelliptic curves $C_1(a),\, C_3(a)$ put forward by Sutherland in \cite[\href{https://swc-math.github.io/aws/2016/2016SutherlandOutline.pdf}{Project 1}]{SutherlandAWS}. 

When $m$ is odd, we obtain a more precise result. \cite[Proposition 5.3.2]{part1} shows that the Galois action on the space of Tate classes is actually \textit{independent} of $a$, and therefore so is the $\ell$-adic monodromy group. Since $\mathcal{G}_{A, \ell}$ determines $\operatorname{ST}(A)$, we conclude that $\operatorname{ST}( \operatorname{Jac}(C_{m,a}))$ is independent of $a$. We state this result formally.
\begin{theorem}\label{thm: twists}
    Let $m \geq 3$ be an odd integer. For $a \in \Q^\times$ let $C_{m,a}$ be the smooth projective curve with affine equation $y^2=x^m+a$.
    The Sato-Tate group of $\operatorname{Jac}(C_{m,a})$ is isomorphic to the Sato-Tate group of $J_m$.
\end{theorem}

\subsection{The polarization of \texorpdfstring{$J_m$}{Jm}}
Like all Jacobian varieties, $J_m$ carries a canonical principal polarization, which we now describe explicitly. %
One has a canonical isomorphism $\HH^1(J_m(\C), \Q) \cong \HH^1(C_m(\C), \Q)$. Since $C_m(\C)$ is a compact surface {with a complex structure} (hence orientable), there is a canonical isomorphism $\HH^2(C_m(\C), \Q) \cong \Q$ (in fact, we even have $\HH^2(C_m(\C), \mathbb{Z}) \cong \Z$). The cup product on $\HH^1(C_m(\C), \Q)$ then gives rise to an anti-symmetric, bilinear map
\[
\HH^1(J_m(\C), \Q) \times \HH^1(J_m(\C), \Q) \cong \HH^1(C_m(\C), \Q) \times \HH^1(C_m(\C), \Q) \xrightarrow{\cup} \HH^2(C_m(\C), \Q) \cong \Q,
\]
which is our polarisation. Tensoring with $\mathbb{C}$ and identifying $\HH^1(J_m(\C), \Q) \otimes \mathbb{C}$ with $\HH^1_{\operatorname{dR}}(J_m(\C), \mathbb{C})$, we obtain an anti-symmetric bilinear form on $\HH^1_{\operatorname{dR}}(J_m(\C), \mathbb{C})$ that we denote by $\varphi$. We describe the matrix of $\varphi$ with respect to the basis $x^{i-1}\,{dx}/{y}$ of \Cref{remark: the hyperelliptic ladic basis}.

\begin{proposition}\label{prop: matrix of the polarisation}
The matrix of the canonical polarisation $\varphi$ in the basis $x^{i-1}\, dx/y$ is
\[
{-\frac{1}{4}}\begin{pmatrix}
0 & 0 & 0 & \cdots & 0 & 0 & \frac{m}{m-2} \\
0 & 0 & 0 & \cdots & 0 & \frac{m}{m-4} & 0 \\
0 & 0 & 0 & \cdots & \frac{m}{m-6} & 0 & 0 \\
  &   & & \iddots \\
  0 & 0 & - \frac{m}{m-6} & \cdots & 0 & 0 & 0 \\
  0 & - \frac{m}{m-4} & 0 & \cdots & 0 & 0 & 0 \\
   - \frac{m}{m-2} & 0 & 0 & \cdots & 0 & 0 & 0
\end{pmatrix}.
\]
\end{proposition}
\begin{proof}
Recall the twisted curve $\Tilde{C}_m$ and its Jacobian $\Tilde{J}_m$ defined in \cite[Section 5.3]{part1}. The twisting isomorphism $t\colon \Tilde{C}_{m, \overline{\Q}} \to C_{m, \overline{\Q}}, \, (x,y)\mapsto(\mu x,y)$, with $\mu^m=-4$, gives a commutative diagram %
\begin{equation}
{\label{eq: twisting the polarization}}
\begin{tikzcd}
    \HH^1(C_m(\C), \C) \times \HH^1(C_m(\C), \C) \arrow[r, "\varphi"] \arrow[d,"t^\ast"]
    & \HH^2(C_m(\C), \C)  \arrow[r, "\cong"]\arrow[d,"t^\ast"]
    & \C\arrow[d, equal] \\
    \HH^1(\Tilde{C}_m(\C), \C) \times \HH^1(\Tilde{C}_m(\C), \C) \arrow[r, "\Tilde{\varphi}"]
    & \HH^2(\Tilde{C}_m(\C), \C)  \arrow[r, "\cong"]
    & \C.
\end{tikzcd}
\end{equation}
The isomorphism $t$ pulls back the differential form $x^{i-1} \, dx/y$ on $C_m$ to the differential form $\mu^{-i} x^{i-1} \, dx/y$ on $\Tilde{C}_m$. In particular
\begin{equation}
    {\label{eq: polarization of J to polarization of the twist}}
    \begin{array}{rl}
         \varphi\left(x^{i-1} \tfrac{dx}{y}, x^{j-1} \tfrac{dx}{y}\right) & =
         \varphi\left(t_\ast(\mu^{-i} x^{i-1} \tfrac{dx}{y}), t_\ast(\mu^{-j} x^{j-1} \tfrac{dx}{y})\right)\\
         & = \mu^{-(i+j)} \cdot  t_\ast\Tilde{\varphi}\left(x^{i-1} \tfrac{dx}{y},x^{j-1} \tfrac{dx}{y}\right).
    \end{array}
\end{equation}
We now compute the polarization on the Jacobian variety $\Tilde{J}_m$ of the twist $\Tilde{C}_m$.
The basis of 1-forms $\{x^i \, dx/y\}$ of $\HH^1_{\operatorname{dR}}(\Tilde{C}_m(\C), \C)$ pulls back, via $X^1_m \to \Tilde{C}_m$, to the set of linearly independents elements $\{\omega_i\}$ of $\HH^1_{\operatorname{dR}}(X_m^1(\C), \C)$ (see Equation \Cref{eq: definition of omegai} and the preceding discussion).
The cup product and the pull-back commute. Therefore, to compute the polarization on $\Tilde{J}_m = \Jac(\Tilde{C}_m)$, it is enough to determine the value of $\tilde{\varphi}(\omega_i, \omega_j) = \omega_i \cup \omega_j$. We computed this product in the proof of \Cref{lemma: Poincare isomorphism}: in particular, \Cref{eq: cup product omegai final result} gives 
\[
\tilde{\varphi}(\omega_i, \omega_j) = \delta_{i, m-j} \cdot \frac{m}{m-2i}.
\]
Replacing in \eqref{eq: polarization of J to polarization of the twist} and using $\mu^{-m} = - \tfrac{1}{4}$, we obtain the result in the statement.
\end{proof}

\subsection{A generalization of the Gross-Koblitz formula}\label{subsect: generalized Gross Koblitz}
In this section we prove a partial generalization of a formula of Gross and Koblitz that relates the values of the $p$-adic and usual Gamma functions when evaluated at suitable arguments. We do not use this result in other parts of the paper, so we only give limited details.

Recall that \textit{Morita's $p$-adic Gamma function} is the unique continuous function
\[
\Gamma_p : \mathbb{Z}_p \to \mathbb{Z}_p
\]
such that $\Gamma_p(x) = (-1)^x \prod_{\substack{0<i<x, \,  p \,\nmid\, i}} i$ for all positive integers $x$. Similarly to what we did for the usual $\Gamma$ function, we extend $\Gamma_p$ on characters by setting, given a character $\alpha=(\alpha_0, \ldots, \alpha_{n+1}) \in (\Z/m\Z)^{n+2}$,
\[
\Gamma_p(\alpha) \colonequals \prod_{i=0}^{n+1} \Gamma_p\left( \{\frac{\alpha_i}{m}\} \right),
\]
where we keep following the convention for the fractional part $\{\cdot\}$ of \cite[Definition 6.1.2]{part1}.

Let us now recall the formula of Gross and Koblitz \cite{GrossKoblitz} that we want to generalize. Notice that our convention for $\Gamma(\alpha)$, given in \cite[Definition 6.1.2]{part1}, differs from that in \cite{GrossKoblitz} by a factor $(2\pi i)^{-\langle \alpha \rangle}$.

\begin{theorem}[Gross-Koblitz, {\cite[Theorem 4.6]{GrossKoblitz}}]\label{thm: Gross-Koblitz}
    Let $m$ be a positive integer, $K=\mathbb{Q}(\zeta_m)$, and $\mathcal{P}$ be a prime of $\mathcal{O}_K$. Suppose that the residue field of $\mathcal{P}$ has prime order $p$. Let $\operatorname{Frob}_{\mathcal{P}}$ be a geometric Frobenius element in $\operatorname{Gal}(\overline{\Q}/K)$ corresponding to $\mathcal{P}$ and let $\iota_\mathcal{P} : K \hookrightarrow K_{\mathcal{P}}$ be the inclusion of $K$ in its $\mathcal{P}$-adic completion. If $\alpha$ is a character such that $\langle u\alpha \rangle$ is independent of $u \in (\Z/m\Z)^\times$, then
    \begin{equation}\label{eq: Gross-Koblitz}
\Gamma(\alpha)^{\operatorname{Frob}_{\mathcal{P}}-1} = \iota_{\mathcal{P}}^{-1} \left( \frac{\Gamma_p(\alpha)}{(-1)^{\langle \alpha \rangle}} \right).
    \end{equation}
\end{theorem}
We aim to prove a version of Equation \Cref{eq: Gross-Koblitz} under the only assumption that $p$ is unramified in $K$, at least for some characters $\alpha$.

Let $m\geq 3$ be an odd integer and $\alpha=(\alpha_0, \ldots, \alpha_{n+1}) =\gamma_{i_1} \ast \cdots \ast \gamma_{i_q}$ be a character obtained as the concatenation of $q$ characters of the form $\gamma_i = (i,i, -2i)$, where each $i$ is an integer in the interval $[1, m-1]$, as already defined in \Cref{remark: character indexing}. Assume that $\langle u\alpha \rangle$ is independent of $u$; this implies that $q$ is even by \cite[Proposition 7.6 and Corollary 7.7]{Deligne}. Fix a prime $p$ that does not divide $m$. We will obtain our version of \eqref{eq: Gross-Koblitz} by comparing two Frobenius actions on the cohomology space $\HH^1(X_m^1)^{\otimes q}\left( \frac{q}{2} \right)$ (for suitable cohomology theories $\HH$).

On the one hand, we use a formula of Coleman \cite[Theorem 19]{Coleman}, who computed the action of the Frobenius $\phi$ on the first de Rham cohomology of the Fermat curve $X^1_m / \Q_p$ and found the relation 
\[
\phi[\nu_{r,r}] = c_{r} [\nu_{pr, pr}],
\]
where $\nu_{r,r}$ is as in the proof of \Cref{lemma: Poincare isomorphism},
\[
c_{r} = \iota_r p^{\varepsilon(-r)} \frac{\Gamma_p(\{2pr\})}{\Gamma_p(\{pr\})^2},
\]
and we have set
\[
\varepsilon(r) = 2\{r\} - \{2r\} \quad \text{ and }
\quad \iota_r = \begin{cases}
    (-1)^{\varepsilon(r)} \text{ if $p$ is odd} \\
    (-1)^{\varepsilon(r) + 2\{2r\}} \text{ if $p=2$}.
\end{cases}
\]
Taking the tensor product $\nu_\alpha \colonequals \nu_{i_1/m, i_1/m} \otimes \cdots \otimes \nu_{i_q/m, i_q/m}$ of differential forms, we get that in $\HH_{\operatorname{dR}}^1(X_m^1/\Q_p)^{\otimes q}$ we have
\[
\phi[\nu_\alpha] =\left(\prod_{r=1}^q c_{i_r/m}\right) [\nu_{p\alpha}].
\]
We now simplify the expression
\begin{equation}\label{eq: action of crystalline Frobenius, Coleman}
    \begin{aligned}
\prod_{j=1}^q c_{ i_j/m } & = \prod_{j=1}^q (\iota_{ i_j/m} \cdot p^{\varepsilon(- i_j/m)}) \cdot \frac{\Gamma_p( [2pi_j]/m)}{\Gamma_p([pi_j]/m)^2} \\
& = \prod_{j=1}^q \iota_{i_j/m} \cdot p^{\sum_{j=1}^q \varepsilon(- i_j/m)} \cdot \prod_{j=1}^q \frac{\Gamma_p([2pi_j]/m)}{\Gamma_p([pi_j]/m)^2}.
\end{aligned}
\end{equation}
For every $u \in (\Z/m\mathbb{Z})^\times$ we have
\begin{equation}\label{eq: sum of epsilons}
\begin{aligned}
\sum_{j=1}^q \varepsilon(ui_j/m) & = \sum_{j=1}^q \left( 2 \{ui_j/m\} -\{2ui_j/m\} \right) \\
& = \sum_{j=1}^q \left( \{ui_j/m\} + \{ui_j/m\} +\{-2ui_j/m\}-1 \right) \\
& = -q + \langle u \alpha \rangle = q/2.
\end{aligned}
\end{equation}
Applying this with $u=1$ and $q$ even, for odd $p$ we obtain
\[
\prod_{j=1}^q \iota_{i_j} = (-1)^{\sum_{j=1}^q \varepsilon(i_j/m)} = (-1)^{-q + \langle \alpha \rangle} = (-1)^{\langle \alpha \rangle}.
\]
One checks that the same equality $\prod_{j=1}^q \iota_{i_j}=(-1)^{\langle \alpha \rangle}$ holds also for $p=2$.
Using \eqref{eq: sum of epsilons} with $u=-1$ we also obtain $p^{\sum_{j=1}^q \varepsilon(-i_j/m)} = p^{q/2}$. Combining these equalities with \eqref{eq: action of crystalline Frobenius, Coleman}, we get 
\[
\prod_{j=1}^q c_{i_j/m} = (-1)^{\langle \alpha \rangle} p^{q/2} \prod_{j=1}^q \frac{\Gamma_p([2pi_j]/m)}{\Gamma_p([pi_j]/m)^2}.
\]
Taking into account the twist by $\frac{q}{2}$, which multiplies the Frobenius by $p^{-q/2}$ (see for example the last page of \cite{MR1265557}),
we see that the Frobenius $\phi$ on $\HH^1_{\operatorname{dR}}(X_m^1/\Q_p)^{\otimes q}\left(\frac{q}{2}\right)$ satisfies
\begin{equation}\label{eq: dR Frobenius}
\phi [\nu_\alpha] = (-1)^{\langle \alpha \rangle} \prod_{j=1}^q \frac{\Gamma_p([2pi_j]/m)}{\Gamma_p([pi_j]/m)^2} [\nu_{p\alpha}]. 
\end{equation}
For simplicity of notation, for $u \in (\Z/m\Z)^\times$ and $\alpha=\gamma_{i_1} \ast \cdots \ast \gamma_{i_q}$ we set
\begin{equation}\label{eq: hat Gammap}
\hat{\Gamma}_p(u \alpha) = \prod_{j=1}^q \frac{\Gamma_p\left( \frac{[ui_j]}{m} \right)^2}{\Gamma_p\left( \frac{[2ui_j]}{m} \right)},
\end{equation}
so that \eqref{eq: dR Frobenius} can be rewritten as
\begin{equation}\label{eq: dR Frobenius 2}
\phi [\nu_\alpha] = (-1)^{\langle \alpha \rangle} \hat{\Gamma}(p\alpha)^{-1} [\nu_{p\alpha}]. 
\end{equation}

On the other hand, we now compute the action of a geometric Frobenius at $p$ on the étale cohomology of $X_m^1$ (considered as a variety defined over $\Q$). Let $\operatorname{Frob}_p \in \operatorname{Gal}(\overline{\Q}/\Q)$ be a geometric Frobenius element at $p$. In particular, $\operatorname{Frob}_p$ sends $\zeta_m$ to $\zeta_m^{p^{-1}}$, so that in the notation of 
\Cref{thm: explicit Galois action on Tate classes} we have $u(\operatorname{Frob}_p)=p^{-1}$. The choice of $\operatorname{Frob}_p$ corresponds to the choice of a prime of $\overline{\Q}$ over $p$, hence of a prime $\mathcal{P}$ of $\Q(\varepsilon_{J_m})$, an embedding $\iota_{\mathcal{P}}$ of $\Q(\varepsilon_{J_m})$ inside its $\mathcal{P}$-adic completion $\Q(\varepsilon_{J_m})_{\mathcal{P}}$, and a compatible embedding $\sigma_p : \overline{\Q} \hookrightarrow \overline{\Q_p}$.

We briefly recall that Fontaine's ring $B_{\operatorname{dR}}$ contains an algebraic closure of $\Q_p$ inside $\mathbb{C}_p$. It follows that we can identify $\sigma_p(\Q(\varepsilon_{J_m})) = \iota_{\mathcal{P}}(\Q(\varepsilon_{J_m}))$ to a subfield of $\overline{\mathbb{Q}_p}$, hence to a subring of $B_{\operatorname{dR}}$.
\begin{remark}
    Not much would change if we used the crystalline ring $B_{\operatorname{cris}}$ instead, provided that $p$ does not divide $m$ (which we assume). The ring $B_{\operatorname{cris}}$ contains the maximal unramified extension of $\Q_p$ inside $\mathbb{C}_p$. The extension $\Q(\varepsilon_{J_m})/\Q$ is unramified away from the divisors of $m$: one has $\Q(\varepsilon_{J_m}) = \bigcap_{p \text{ prime}} \Q(J_m[p^\infty])$ by \cite[Theorem 0.1]{MR1441234}, and by the Néron-Ogg-Shafarevich criterion the field $\bigcap_{p \text{ prime}} \Q(J_m[p^\infty])$ is easily seen to be ramified at most at the primes of bad reduction of $J_m$, which divide $m$. We can then identify $\sigma_p(\Q(\varepsilon_{J_m})) = \iota_{\mathcal{P}}(\Q(\varepsilon_{J_m}))$ to a subfield of $\mathbb{Q}_p^{\operatorname{nr}}$, hence to a subring of $B_{\operatorname{cris}}$. Note in particular that since $\Q(\varepsilon_{J_m})/\Q$ is unramified at $p$, the choice of a Frobenius at $p$ for this extension is exactly equivalent to the choice of a prime of $\Q(\varepsilon_{J_m})$ lying above $p$, which we have already implicitly claimed above.
\end{remark}

By a well-known comparison theorem \cite{MR1463696}, there is a canonical isomorphism
\[
\HH^q_{\operatorname{\acute{e}t}}\left( (X_{m, \overline{\Q}}^1)^{q}, \Q_p(\tfrac{q}{2}) \right) \otimes_{\Q_p} B_{\operatorname{dR}} \cong \HH^q_{\operatorname{dR}}\left( (X_m^1)^{q}  /\Q_p\right) \left( \tfrac{q}{2}\right)\otimes_{\Q_p}B_{\operatorname{dR}},
\]
and the formula of \Cref{thm: explicit Galois action on Tate classes} makes sense in this space. In particular, \Cref{thm: explicit Galois action on Tate classes} shows that, under the above identification, $\operatorname{Frob}_p$ sends 
$[\omega_\alpha] \otimes 1$ to 
\[ [\omega_{p\alpha}] \otimes \iota_{\mathcal{P}}\left(\frac{\mu_{-p\alpha}}{\mu_{-\alpha}} \cdot \frac{\operatorname{Frob}_p(P(\gamma, \omega_{-p\alpha}))}{P(\gamma, \omega_{-\alpha})}\right).
\]
Now recall that Coleman's basis $\nu_{r,r}$ is related to our basis $\omega_i$ (up to certain identifications that we have spelled out in the proof of \Cref{lemma: Poincare isomorphism}) by the rescaling
\[
\nu_{i/m, i/m} = K\left( \frac{i}{m}, \frac{i}{m} \right) \omega_i.
\]
Write
\[
e_i = K\left( \frac{i}{m}, \frac{i}{m} \right) = \begin{cases}
\frac{m-2i}{m}, \text{ if } i>\frac{m}{2} \\
1, \text{ otherwise}.
\end{cases}
\]
We obtain
\[
\nu_\alpha = \bigotimes_{j=1}^q \nu_{i_j/m, i_j/m} = \bigotimes_{j=1}^q e_{i_j} \omega_{i_j} = \left(\prod_{j=1}^q e_{i_j}\right) \omega_\alpha. 
\]
Hence, $\operatorname{Frob}_p$ satisfies
\begin{equation}\label{eq: action of geometric Frobenius}
\operatorname{Frob}_p\left([\nu_\alpha] \otimes 1\right) = [\nu_{p\alpha}] \otimes \frac{\prod_{j=1}^q e_{i_j}}{\prod_{j=1}^q e_{pi_j \bmod m}} \cdot \frac{\mu_{-p\alpha}}{\mu_{-\alpha}} \cdot 
\iota_{\mathcal{P}}\left(\frac{\operatorname{Frob}_p(P(\gamma, \omega_{-p\alpha}))}{P(\gamma, \omega_{-\alpha})} \right).
\end{equation}

We will now manipulate this expression to make it more explicit. Recall from \eqref{eq: evaluation of P gamma omega-alpha} that
\[
P(\gamma, \omega_{-\alpha})^{-1} = (2\pi i)^{q/2} \prod_{j=1}^q \Gamma\left( \frac{[-i_j]}{m} \right)^{-2} \Gamma\left( \frac{2[-i_j]}{m} \right).
\]
Observe that $2[-i_j]=2(m-i_j)$. In particular,
\begin{enumerate}
    \item if $i_j < \frac{m}{2}$ we have $e_{i_j} = 1$ and
    \[
    \begin{aligned}
    \Gamma\left( \frac{2[-i_j]}{m} \right) & = \Gamma\left( \frac{m+(m-2i_j)}{m} \right)= \Gamma\left( 1+ \frac{m-2i_j}{m} \right) \\ &= \frac{m-2i_j}{m} \Gamma\left( \frac{m-2i_j}{m} \right) = \frac{m-2i_j}{m} \Gamma\left( \frac{[-2i_j]}{m} \right);
    \end{aligned}
    \]
    \item if $i_j>\frac{m}{2}$ we have $e_{i_j} = \frac{m-2i_j}{m}$ and
    \[
    \Gamma\left( \frac{2[-i_j]}{m} \right) = \Gamma\left( \frac{2m-2i_j}{m} \right)= \Gamma\left( \frac{[-2i_j]}{m} \right).
    \]
\end{enumerate}
Thus, since $\mu_{i_j}=\frac{m-2i_j}{m}$, we see that for every $j$ we have the uniform formula
\[
e_{i_j} \Gamma\left( \frac{[-i_j]}{m} \right)^{-2} \Gamma\left( \frac{2[-i_j]}{m} \right) = \mu_{i_j}\Gamma\left( \frac{[-i_j]}{m} \right)^{-2} \Gamma\left( \frac{[-2i_j]}{m} \right).
\]
Taking the product over $j$, and applying a similar argument to $P(\gamma, \omega_{-p\alpha})$, we find that \eqref{eq: action of geometric Frobenius} yields
\begin{equation}\label{eq: action of geometric Frobenius 2}
\operatorname{Frob}_p\left([\nu_\alpha] \otimes 1\right) = [\nu_{p\alpha}] \otimes \frac{\mu_{\alpha}}{\mu_{p\alpha}} \cdot \frac{\mu_{-p\alpha}}{\mu_{-\alpha}} \cdot 
\iota_{\mathcal{P}}\left(\frac{\operatorname{Frob}_p\left( \hat{\Gamma}(-p\alpha)
\right)}{ \hat{\Gamma}(-\alpha)}
\right),
\end{equation}
where we have set, for $u \in (\Z/m\Z)^\times$ and $\alpha=\gamma_{i_1} \ast \cdots \ast \gamma_{i_q}$,
\begin{equation}\label{eq: hat Gamma}
    \hat{\Gamma}(u\alpha) = (2\pi i)^{-q/2} \prod_{j=1}^q \frac{\Gamma\left( \frac{[ui_j]}{m} \right)^2}{\Gamma\left( \frac{[2ui_j]}{m} \right)}.
\end{equation}
Recalling \Cref{rmk: mu alpha is an even function in alpha}, Equation \eqref{eq: action of geometric Frobenius 2} then simplifies to
\begin{equation}\label{eq: action of geometric Frobenius 3}
\operatorname{Frob}_p\left([\nu_\alpha] \otimes 1\right) = [\nu_{p\alpha}] \otimes 
\iota_{\mathcal{P}}\left(\frac{\operatorname{Frob}_p\left( \hat{\Gamma}(-p\alpha)
\right)}{ \hat{\Gamma}(-\alpha)}
\right).
\end{equation}

We can now compare \eqref{eq: dR Frobenius 2} and \eqref{eq: action of geometric Frobenius 3}. By \cite[Corollary 4.10]{Ogus} (which gives the result for all but finitely many primes) or \cite[Theorem 5.3]{MR1265557} (which gives the result for all primes of good reduction of $\operatorname{Jac}(X_m^1)$, that is, all primes not dividing $m$), we know that the action of the crystalline Frobenius coincides with the action of the geometric Frobenius on the space of absolute Hodge classes. Note that the quoted theorems give this result for abelian varieties, but  $\HH^1(X_m^1)^{\otimes q}$ is isomorphic to $\HH^1(\operatorname{Jac}(X_m^1))^{\otimes q}$, which in turn embeds into $\HH^q(\operatorname{Jac}(X_m^1)^q)$ by K\"unneth's formula -- and this for any cohomology theory $\HH$ that we need to use. Since we know from \Cref{sub: strategy} that the $[\alpha]$-eigenspace of $\HH^1(X_m^1(\C),\Q)^{\otimes q}\left( \frac{q}{2} \right)$ consists of absolute Hodge classes, these results apply: comparing Equations \eqref{eq: dR Frobenius 2} and \eqref{eq: action of geometric Frobenius 3}, we obtain
\[
(-1)^{\langle \alpha \rangle} \hat{\Gamma}_p(p\alpha)^{-1}
= 
\iota_{\mathcal{P}}\left(\frac{\operatorname{Frob}_p\left( \hat{\Gamma}(-p\alpha)
\right)}{ \hat{\Gamma}(-\alpha)}
\right).
\]

We can then apply $\iota_{\mathcal{P}}^{-1}$ to both sides of the previous equality to obtain our version of the Gross-Koblitz formula:
\generalizedGK

\begin{remark}
    We have obtained this formula for characters of the form $\alpha = \gamma_{i_1} \ast \cdots \ast \gamma_{i_q}$. It seems very likely that a more general calculation in the cohomology of a suitable Fermat hypersurface $X_m^n$ could give an analogue of the formula of \Cref{thm: Gross-Koblitz general version} for all characters considered in \Cref{thm: Gross-Koblitz} and for all primes $p$ that do not divide $m$.
\end{remark}

\begin{remark}
    When $p \equiv 1 \pmod m$, the formula in the theorem simplifies to
\[
\frac{\hat{\Gamma}(-\alpha)}{\operatorname{Frob}_p(\hat{\Gamma}(-\alpha))} = \iota_{\mathcal{P}}^{-1}\left( \frac{\hat{\Gamma}_p(\alpha)}{(-1)^{\langle \alpha \rangle}}\right).
\]
Up to factors in $\Q(\zeta_m)$ we have $\hat{\Gamma}(-\alpha) = \hat{\Gamma}(\alpha)^{-1}$: this is easy to check using the reflection formula and the fact that $q$ is even.
Since $p \equiv 1 \pmod{m}$ is totally split in $\Q(\zeta_m)$, the Frobenius at $p$ acts trivially on $\Q(\zeta_m)$, and therefore these factors in $\Q(\zeta_m)$ cancel out when taking the ratio on the left-hand side. Thus, in the case $p \equiv 1 \pmod m$ we obtain
\[
\frac{\operatorname{Frob}_p(\hat{\Gamma}(\alpha))}{\hat{\Gamma}(\alpha)} = \iota_{\mathcal{P}}^{-1}\left( \frac{\hat{\Gamma}_p(\alpha)}{(-1)^{\langle \alpha \rangle}}\right),
\]
which essentially gives back \eqref{eq: Gross-Koblitz} for the characters we are considering. %
\end{remark}

\begin{remark}
    As the proof shows, once we have the results of Coleman \cite{Coleman} on the one hand, and Blasius \cite{MR1265557} and Ogus \cite{Ogus} on the other, \Cref{thm: Gross-Koblitz general version} is essentially equivalent to \Cref{thm: explicit Galois action on Tate classes}. The advantage is that \Cref{thm: Gross-Koblitz general version} makes no reference to cohomology, which makes it much easier to test numerically. We have tested the formula of \Cref{thm: Gross-Koblitz general version} on many examples and found it to be correct in all the cases we tried.
\end{remark}

\bibliographystyle{alpha}
\bibliography{biblio}

\end{document}